\documentclass[]{article}

\usepackage{amsmath,amsthm,amssymb,enumerate,epsfig,color,graphicx,epstopdf,tikz,rotating}

\usetikzlibrary{positioning,shapes}

\usepackage[all]{xy}

\usepackage{array}
\newcolumntype{C}[1]{>{\centering\arraybackslash$}p{#1}<{$}}

\usepackage{ifdraft}
\usepackage[colorlinks=true,linkcolor=blue,draft=false]{hyperref}
\usepackage{aliascnt}

\def\<{\langle}
\def\>{\rangle}

\def\N{{\mathbb N}}
\def\Z{{\mathbb Z}}

\theoremstyle{plain}
\newtheorem{theorem}{Theorem}[section]

\newaliascnt{lemma}{theorem}
\newtheorem{lemma}[lemma]{Lemma}
\aliascntresetthe{lemma}

\newaliascnt{proposition}{theorem}
\newtheorem{proposition}[proposition]{Proposition}
\aliascntresetthe{proposition}

\newaliascnt{corollary}{theorem}
\newtheorem{corollary}[corollary]{Corollary}
\aliascntresetthe{corollary}

\newaliascnt{conjecture}{theorem}
\newtheorem{conjecture}[conjecture]{Conjecture}
\aliascntresetthe{conjecture}

\theoremstyle{remark}

\newaliascnt{claim}{theorem}

\aliascntresetthe{claim}

\newtheorem*{claim*}{Claim}
\newtheorem*{remark}{Remark}

\theoremstyle{definition}

\newaliascnt{definition}{theorem}
\newtheorem{definition}[definition]{Definition}
\aliascntresetthe{definition}

\newaliascnt{example}{theorem}
\newtheorem{example}[example]{Example}
\aliascntresetthe{example}

\newaliascnt{notation}{theorem}

\aliascntresetthe{notation}

\ifdraft{
\newcommand{\colorV}[1]{{\color{blue}#1}}
\newcommand{\colorJ}[1]{{\color{red} #1}}
\newcommand{\colorM}[1]{{\color[rgb]{0,0.7,0} #1}}
\newcommand{\colorB}[1]{{\color[rgb]{0.8,0.2,0} #1}}
\newcommand{\commV}[1]{\marginpar{\tiny\vskip-3ex\colorV{V: #1}}}
\newcommand{\commJ}[1]{\marginpar{\tiny\vskip-3ex\colorJ{J: #1}}}
\newcommand{\commM}[1]{\marginpar{\tiny\vskip-3ex\colorM{M: #1}}}
\newcommand{\commB}[1]{\marginpar{\tiny\vskip-3ex\colorB{B: #1}}}
\newcommand{\comment}[1]{\marginpar{\tiny\vskip-3ex\color[rgb]{0.5,0.5,0.5} #1}}
}
{
\newcommand{\colorV}[1]{#1}
\newcommand{\colorJ}[1]{#1}
\newcommand{\colorM}[1]{#1}
\newcommand{\colorB}[1]{#1}
\newcommand{\commV}[1]{}
\newcommand{\commJ}[1]{}
\newcommand{\commM}[1]{}
\newcommand{\commB}[1]{}
\newcommand{\comment}[1]{}
}

\title{On parabolic subgroups of Artin--Tits groups\\of spherical type}
\author{\colorM{Mar\'{\i}a Cumplido}, \colorV{Volker Gebhardt}, \colorJ{Juan Gonz\'alez-Meneses}\\ and \colorB{Bert Wiest}\footnote{Partially supported by a PhD contract funded by University of Rennes 1, by the Ministry of Economy and Competitiveness of Spain and FEDER through the projects MTM2013-44233-P, MTM2016-76453-C2-1-P, by the Ministry of Education, Culture and Sports of Spain and the French Embassy in Spain through the mobility program ``M\'{e}rim\'{e}e 2015" DCT2015/0002, and by Western Sydney University through a honorary academic appointment as Visiting Professor. Part of this work was done during a visit of the third author to Western Sydney University, and visits of the first, third and fourth authors to the University of Seville and the University of Rennes 1.}}
\date{June 19, 2019}

\begin{document}

\maketitle


\begin{abstract}
We show that, in an Artin--Tits group of spherical type, the intersection of two parabolic subgroups is a parabolic subgroup.
Moreover, we show that the set of parabolic subgroups forms a lattice with respect to inclusion.
This extends to all Artin--Tits groups of spherical type a result that was previously known for braid groups.

To obtain the above results, we show that every element in an Artin--Tits group of spherical type admits a unique minimal parabolic subgroup containing it, which we call its parabolic closure.
We also show that the parabolic closure of an element coincides with the parabolic closure of any of its powers or roots.
As a consequence, if an element belongs to a parabolic subgroup, all its roots belong to the same parabolic subgroup.

We define the simplicial complex of irreducible parabolic subgroups, and we propose it as the analogue, in Artin--Tits groups of spherical type, of the celebrated complex of curves which is an important tool in braid groups, and more generally in mapping class groups.
We conjecture that the  complex of irreducible parabolic subgroups is $\delta$-hyperbolic.
\end{abstract}

\section{Introduction}\label{S:Introduction}

Artin--Tits groups are a natural generalization of braid groups from the algebraic point of view:
In the same way that the braid group can be obtained from the presentation of the symmetric group with transpositions as generators by dropping the order relations for the generators, other Coxeter groups give rise to more general Artin--Tits groups.
If the underlying Coxeter group is finite, the resulting Artin--Tits group is said to be of \emph{spherical type}.

Artin--Tits groups of spherical type share many properties with braid groups.
For instance, they admit a particular algebraic structure, called \emph{Garside structure}, which allows to define normal forms, solve the word and conjugacy problems, and prove some other algebraic properties (like torsion-freeness).

However, some properties of braid groups are proved using topological or geometrical techniques, since a braid group can be seen as the fundamental group of a configuration space, and also as a mapping class group of a punctured disc.
As one cannot replicate these topological or geometrical techniques in other Artin--Tits groups, they must be replaced by algebraic arguments, if one tries to extend properties of braid groups to all Artin--Tits groups of spherical type.

In this paper we will deal with parabolic subgroups of Artin--Tits groups.
A parabolic subgroup is by definition the conjugate of a subgroup generated by a subset of the standard generators.
The irreducible parabolic subgroups, as we will see, are a natural algebraic analogue of the isotopy classes of non-degenerate simple closed curves in the punctured disc.

The isotopy classes of simple closed curves, in turn, are the building blocks that form the well-known {\it complex of curves}.
The properties of this complex, and the way in which the braid group acts on it, allow to use geometric arguments to prove important results about braid groups (Nielsen--Thurston classification, structure of centralizers, etc.).
Similarly, the set of irreducible parabolic subgroups also forms a simplicial complex, which we call the {\em complex of irreducible parabolic subgroups}.
We conjecture that the geometric properties of this complex, and  the way in which an Artin--Tits group acts on it, will allow to extend many of the mentioned properties to all Artin--Tits groups of spherical type.
The results in the present paper should help to lay the foundations for this study.
\bigskip

In this paper we will show the following:

\begin{theorem}\label{T:minimal_parabolic_containing_an_element}
(see \autoref{P:minimalContainingParabolic}) Let $A_S$ be an Artin--Tits group of spherical type, and let $\alpha\in A_S$. There is a unique parabolic subgroup~$P_\alpha$ which is minimal (by inclusion) among all parabolic subgroups containing~$\alpha$. We call it the {\it parabolic closure} of~$\alpha$.
\end{theorem}

The name {\it parabolic closure} comes from the theory of Coxeter groups.
In a Coxeter group, the parabolic closure of a set is the intersection of all parabolic subgroups containing the set.
It is known~\cite{Qi} that, if the Coxeter group has finite rank, this intersection is a parabolic subgroup, so it is the smallest parabolic subgroup (by inclusion) containing the set.

We will also show that the parabolic closure of an element coincides with the parabolic closure of any of its powers and roots:

\medskip
{\bf \autoref{T:parabolic_for_powers_and_roots}.}
{\em Let $A_S$ be an Artin--Tits group of spherical type. If $\alpha\in A_S$ and $m$ is a nonzero integer, then $P_{\alpha^m}=P_\alpha$.}
\medskip

The above result will have an interesting consequence:

\medskip
{\bf \autoref{C:roots_in_parabolic}.}
{\em Let $A_S$ be an Artin--Tits group of spherical type. If $\alpha$ belongs to a parabolic subgroup~$P$, and $\beta\in A_S$ is such that $\beta^m=\alpha$ for some nonzero integer~$m$, then $\beta\in P$.}
\medskip

Finally, we will use~\autoref{T:minimal_parabolic_containing_an_element} to show the following results, which describe the structure of the set of parabolic subgroups with respect to the partial order determined by inclusion.

\medskip
{\bf \autoref{T:intersection_of_parabolic_is_parabolic}.}
{\em Let~$P$ and~$Q$ be two parabolic subgroups of an Artin--Tits group~$A_S$ of spherical type. Then $P\cap Q$ is also a parabolic subgroup.}
\medskip

{\bf \autoref{T:lattice}.}
{\em The set of parabolic subgroups of an Artin--Tits group of spherical type is a lattice with respect to the partial order determined by inclusion.}
\medskip

{\bf Acknowledgements:} The authors thank Luis Paris for his interesting suggestions.

\section{Complex of irreducible parabolic subgroups}\label{S:ComplexOfParabolics}

An Artin--Tits group $A_S$ is a group generated by a finite set~$S$, that admits a presentation with at most one relation for each pair of generators $s,t\in S$ of the form $sts\cdots = tst\cdots$, where the length (which may be even or odd) of the word on the left hand side is equal to the length of the word on the right hand side. We will denote this length $m(s,t)$, and we will say, as usual, that $m(s,t)=\infty$ if there is no relation involving $s$ and $t$. We can also assume that $m(s,t)>1$, otherwise we could just remove one generator.
The two sides of the relation $sts\cdots = tst\cdots$ are expressions for the least common multiple of the generators~$s$ and~$t$ with respect to the prefix and suffix partial orders (cf.\ \autoref{S:Garside}).

The given presentation of~$A_S$ is usually described by a labeled graph~$\Gamma_S$ (called the {\it Coxeter graph} of~$A_S$) whose vertices are the elements of~$S$, in which there is an edge joining two vertices $s$ and $t$ if and only if $m(s,t)>2$. If $m(s,t)>3$ we write $m(s,t)$ as a label on the corresponding edge.

Given an Artin--Tits group~$A_S$, where~$S$ is the standard set of generators and~$\Gamma_S$ is the associated Coxeter graph, we say that~$A_S$ is \emph{irreducible} if~$\Gamma_S$ is connected.
For a subset~$X$ of~$S$, the subgroup~$A_X$ of~$A_S$ generated by~$X$ is called a \emph{standard parabolic subgroup} of~$A_S$; it is isomorphic to the Artin--Tits group associated to the subgraph~$\Gamma_X$ of~$\Gamma_S$ spanned by~$X$~\cite{vanderLek}.
A subgroup~$P$ of~$A_S$ is a \emph{parabolic subgroup} of~$A_S$ if it is conjugate to a standard parabolic~$A_X$ subgroup of~$A_S$; we say that~$P$ is \emph{irreducible} if~$A_X$ is irreducible.

If one adds the relations $s^2=1$ for all $s\in S$ to the standard presentation of an Artin--Tits group~$A_S$, one obtains its corresponding Coxeter group~$W_S$. If~$W_S$ is a finite group, then~$A_S$ is said to be of {\it spherical type}.
Artin--Tits groups of spherical type are completely classified~\cite{Coxeter},
and they are known to admit a Garside structure, as we will see in \autoref{S:Garside}.

The main example of an Artin--Tits group of spherical type is the braid group on $n$~strands, $B_n$, which is generated by $n-1$ elements, $\sigma_1,\ldots, \sigma_{n-1}$, with defining relations $\sigma_i\sigma_j=\sigma_j\sigma_i$ (if $|i-j|>1$) and $\sigma_i\sigma_j\sigma_i=\sigma_j\sigma_i\sigma_j$ (if $|i-j|=1$).
Its corresponding Coxeter group is the symmetric group of $n$ elements,~$\Sigma_n$.

The braid group~$B_n$ can be seen as the group of orientation-preserving automorphisms of the $n$-times punctured disc~$D_n$, fixing the boundary pointwise, up to isotopy.
That is, $B_n$ is the mapping class group of~$D_n$.
As a consequence, $B_n$ acts naturally on the set of isotopy classes of non-degenerate, simple closed curves in~$D_n$. (Here {\it non-degenerate} means that the curve encloses more than one and less than $n$~punctures; {\it simple} means that the curve does not cross itself.)
The (isotopy classes of) curves form a simplicial complex, called the {\it complex of curves}, as follows: A simplex of dimension~$d$ is a set of $d+1$ (isotopy classes of) curves which admit a realisation consisting of $d+1$ mutually disjoint curves.
The 1-skeleton of the complex of curves is called the {\it graph of curves}, in which the vertices are the (isotopy classes of) curves, and two vertices are connected by an edge if and only if the corresponding isotopy classes of curves can be represented by two disjoint curves.

Our goal is to extend the notion of {\it complex of curves} from the braid group~$B_n$ to all Artin--Tits groups of spherical type.
Hence, we need to find some algebraic objects which can be defined for all Artin--Tits groups of spherical type, and which correspond to isotopy classes of non-degenerate simple closed curves in the case of~$B_n$.
We claim that the {\it irreducible parabolic subgroups} are such objects.

We can define a correspondence:
$$
    \varphi:\: \left\{\begin{array}{c}\mbox{Isotopy classes of} \\ \mbox{non-degenerate simple} \\ \mbox{closed curves in }D_n\end{array}\right\} \longrightarrow \left\{\begin{array}{c}\mbox{Proper} \\ \mbox{irreducible parabolic} \\ \mbox{subgroups of }B_n \end{array}\right\}
$$
so that, for a curve~$C$, its image~$\varphi(C)$ is the set of braids that can be represented by an automorphism of~$D_n$ whose support is enclosed by~$C$.

Let us see that the image of~$\varphi$ consists of proper irreducible parabolic subgroups.
Suppose that~$D_n$ is represented as a subset of the complex plane, whose boundary is a circle, and its punctures correspond to the real numbers $1,\ldots,n$.
Let~$C_m$ be a circle enclosing the punctures $1,\ldots, m$ where $1<m<n$. Then~$\varphi(C_m)$ is a subgroup of~$B_n$ which is naturally isomorphic to~$B_m$.
Actually, it is the standard parabolic subgroup~$A_{X_m}$, where $X_m=\{\sigma_1,\ldots,\sigma_{m-1}\}$.
More generally, if we consider a non-degenerate simple closed curve~$C$ in~$D_n$, we can always find an automorphism~$\alpha$ of~$D_n$ such that $\alpha(C)=C_m$ for some $1<m<n$.
This automorphism~$\alpha$ represents a braid, and it follows by construction that~$\varphi(C)$ is precisely the irreducible parabolic subgroup $\alpha A_{X_m} \alpha^{-1}$, which is proper as $1<m<n$.

The correspondence~$\varphi$ is surjective as, given an irreducible parabolic subgroup $P=\alpha A_X \alpha^{-1}$, we can take a circle~$C'$ in~$D_n$ enclosing the (consecutive) punctures which involve the generators in~$X$, and it follows that $\varphi(\alpha^{-1}(C'))=P$.
The injectivity of~$\varphi$ will be shown later as a consequence of \autoref{L:parabolic=central_element}.

Therefore, instead of talking about curves, in an Artin--Tits group~$A_S$ of spherical type we will talk about irreducible parabolic subgroups. The group~$A_S$ acts (from the right) on the set of parabolic subgroups by conjugation. This action corresponds to the action of braids on isotopy classes of non-degenerate simple closed curves.

Now we need to translate the notion of curves being ``disjoint'' (the adjacency condition in the complex of curves) to the algebraic setting. It is worth mentioning that disjoint curves do not necessarily correspond to irreducible parabolic subgroups with trivial intersection.
Indeed, two disjoint nested curves correspond to two subgroups with non-trivial intersection, one containing the other.
Conversely, two irreducible parabolic subgroups with trivial intersection may correspond to non-disjoint curves.
(For instance, the curves corresponding to the two cyclic subgroups of~$B_n$ generated by~$\sigma_1$ and~$\sigma_2$ respectively intersect.)

One simple algebraic translation of the notion of ``disjoint curves" is the following: Two distinct irreducible parabolic subgroups~$P$ and~$Q$ of~$B_n$ correspond to disjoint curves if and only if one of the following conditions is satisfied:
\begin{enumerate}
 \item $P\subsetneq Q$,
 \item $Q\subsetneq P$,
 \item $P\cap Q=\{1\}$ and $pq=qp$ for every $p\in P$, $q\in Q$.
\end{enumerate}
This can be deduced easily using geometrical arguments, but it will also follow from the forthcoming results in this section. Hence, we can say that two irreducible parabolic subgroups~$P$ and~$Q$ in an Artin--Tits group~$A_S$ are {\it adjacent} (similarly to what we say in the complex of curves) if~$P$ and~$Q$ satisfy one of the three conditions above, that is, if either one is contained in the other, or they have trivial intersection and commute.

However, this characterization is not completely satisfactory, as it contains three different cases. Fortunately, one can find a much simpler equivalent characterization by considering a special element for each parabolic subgroup, as we will now see.

If~$P$ is an irreducible parabolic subgroup of an Artin--Tits group of spherical type, we saw that~$P$ is itself an irreducible Artin--Tits group of spherical type.
Hence, the center of~$P$ is cyclic, generated by an element~$z_P$.
To be precise, there are two possible elements ($z_P$ and~$z_P^{-1}$) which could be taken as generators of the center of~$P$, but we take the unique one which is conjugate to a {\it positive} element of~$A_S$ (an element which is a product of standard generators).
Hence, the element~$z_P$ is determined by~$P$.
We shall see in \autoref{L:parabolic=central_element} (case $x=1$), that conversely, the element~$z_P$ determines the subgroup~$P$.

If~$P$ is standard, that is, if~$P=A_X$ for some~$X\subseteq S$, we will write~$z_X=z_{A_X}$.
It turns out that, using the standard Garside structure of~$A_S$, one has $z_X=(\Delta_X)^{e}$, where~$\Delta_X$ is the least common multiple of the elements of~$X$ and $e\in\{1,2\}$. Namely,~$e=1$ if~$\Delta_X$ is central in~$A_X$, and~$e=2$ otherwise~\cite{BS}.

If a standard parabolic subgroup~$A_X$ is not irreducible, that is, if~$\Gamma_X$ is not connected, then $X=X_1\sqcup\cdots \sqcup X_r$, where each~$\Gamma_{X_i}$ corresponds to a connected component of~$\Gamma_X$, and $A_{X}=A_{X_1}\times \cdots \times A_{X_r}$.
In this case we can also define~$\Delta_X$ as the least common multiple of the elements of~$X$, and it is clear that~$(\Delta_X)^e$ is central in~$A_X$ for either~$e=1$ or~$e=2$.
We will define~$z_X=\Delta_X$ if~$\Delta_X$ is central in~$A_X$, and~$z_X=\Delta_X^2$ otherwise. Notice that~$z_X$ is a central element in~$A_X$, but it is not a generator of the center of~$A_X$ if~$A_X$ is reducible, as in this case the center of~$A_X$ is not cyclic: It is the direct product of the centers of each component, so it is isomorphic to~$\mathbb Z^r$.

Now suppose that~$P$ is a parabolic subgroup, so $P=\alpha A_X \alpha^{-1}$ for some $X\subseteq S$.
We define $z_P=\alpha z_X \alpha^{-1}$.
This element~$z_P$ is well defined: If $P=\alpha A_X \alpha^{-1}=\beta A_Y\beta^{-1}$, then we have $\beta^{-1}\alpha A_X \alpha^{-1}\beta =A_Y$, so $\beta^{-1}\alpha z_X \alpha^{-1}\beta = z_Y$ hence $\alpha z_X \alpha^{-1} = \beta z_Y \beta^{-1}$.
Using results by Godelle~\cite{God} one can deduce the following:


\begin{lemma}\label{L:parabolic=central_element}{\rm \cite[Lemma~8]{Cum}}
Let~$P$ and~$Q$ be two parabolic subgroups of~$A_S$. Then, for every $x\in A_S$, one has $x^{-1}Px = Q$ if and only if $x^{-1} z_P x = z_Q$.
\end{lemma}

It follows from the above lemma that, if we want to study elements which conjugate a parabolic subgroup~$P$ to another parabolic subgroup~$Q$, we can replace~$P$ and~$Q$ with the elements~$z_P$ and~$z_Q$. It will be much easier to work with elements than to work with subgroups. Moreover, in the case in which $P=Q$, the above result reads:
$$
   N_{A_S}(P)=Z_{A_S}(z_P),
$$
that is, the normalizer of~$P$ in~$A_S$ equals the centralizer of the element~$z_P$ in~$A_S$.

\begin{remark} We can use \autoref{L:parabolic=central_element} to prove that the correspondence~$\varphi$ is injective. In the case of braid groups, if~$C$ is a non-degenerate simple closed curve, and $P=\varphi(C)$ is its corresponding irreducible parabolic subgroup, the central element~$z_P$ is either a conjugate of a standard generator (if~$C$ encloses two punctures) or the Dehn twist along the curve~$C$ (if~$C$ encloses more than two punctures). Hence, if~$C_1$ and~$C_2$ are such that $\varphi(C_1)=\varphi(C_2)$ then either~$z_P$ or~$z_P^2$ is the Dehn twist along~$C_1$ and also the Dehn twist along~$C_2$. Two Dehn twists along non-degenerate curves correspond to the same mapping class if and only if their corresponding curves are isotopic~\cite[Fact 3.6]{FM}, hence~$C_1$ and~$C_2$ are isotopic, showing that~$\varphi$ is injective.
\end{remark}

\autoref{L:parabolic=central_element} allows us to simplify the {\it adjacency} condition for irreducible parabolic subgroups, using the special central elements~$z_P$:

\begin{theorem}\label{centers_commute=3_conditions}
Let~$P$ and~$Q$ be two distinct irreducible parabolic subgroups of an Artin--Tits group~$A_S$ of spherical type.
Then $z_Pz_Q=z_Qz_P$ holds if and only if one of the following three conditions is satisfied:
\begin{enumerate}

\item $P\subsetneq Q$.

\item $Q\subsetneq P$.

\item $P\cap Q=\{1\}$ and $xy=yx$ for every $x\in P$ and $y\in Q$.

\end{enumerate}
\end{theorem}

The proof of this result will be postponed to \autoref{S:annexe}, as it uses some ingredients which will be introduced later in the paper.

\begin{remark}
Consider two isotopy classes of non-degenerate simple closed curves~$C_1$ and~$C_2$ in the disc~$D_n$, and their corresponding parabolic subgroups $P=\varphi(C_1)$ and $Q=\varphi(C_2)$ of~$B_n$. Then~$C_1$ and~$C_2$ can be realized to be disjoint if and only if their corresponding Dehn twists commute~\cite[Fact 3.9]{FM}. It is known that the centralizer of a generator~$\sigma_i$ is equal to the centralizer of~$\sigma_i^2$. Hence~$C_1$ and~$C_2$ can be realized to be disjoint if and only if~$z_P$ and~$z_Q$ commute, which is equivalent, by \autoref{centers_commute=3_conditions}, to the three conditions in its statement.
\end{remark}

We can finally extend the notion of {\it complex of curves} to all Artin--Tits groups of spherical type, replacing curves with proper irreducible parabolic subgroups.

\begin{definition}
Let~$A_S$ be an Artin--Tits group of spherical type. We define the {\em complex of irreducible parabolic subgroups} as a simplicial complex in which a simplex of dimension~$d$ is a set $\{P_0,\ldots,P_d\}$ of proper irreducible parabolic subgroups such that~$z_{P_i}$ commutes with~$z_{P_j}$ for every $0\leq i,j\leq d$.
\end{definition}

Notice that, by definition, the complex of irreducible parabolic subgroups is a {\it flag complex}, so all the information about the complex is contained in its 1-skeleton, the {\it graph of irreducible parabolic subgroups}.

As it happens with the complex of curves in the punctured disc (or in any other surface), we can define a distance in the complex of irreducible parabolic subgroups, which is determined by the distance in the 1-skeleton, imposing all edges to have length 1. Notice that the action of~$A_S$ on the above complex (by conjugation of the parabolic subgroups) is an action by isometries, as conjugation preserves commutation of elements.

We believe that this complex can be an important tool to study properties of Artin--Tits groups of spherical type.
A proof of the following conjecture
would allow to extend many properties of braid groups to Artin--Tits groups of spherical type (see~\cite{CW} for more context):

\begin{conjecture}
The complex of irreducible parabolic subgroups of an Artin--Tits group of spherical type is $\delta$--hyperbolic.
\end{conjecture}

\section{Results from Garside theory for Artin--Tits groups of spherical type}\label{S:Garside}

In this section, we will recall some results from Garside theory that will be needed.
In order to simplify the exposition, we present the material in the context of Artin--Tits groups of spherical type that is relevant for the paper, instead of in full generality.
For details, we refer to \cite{Deh,DDGKM,DP,ECHLPT}.

Let~$A_S$ be an Artin--Tits group of spherical type. Its monoid of positive elements (the monoid generated by the elements of~$S$) will be denoted~$A_S^+$.
The group~$A_S$ forms a lattice with respect to the prefix order~$\preccurlyeq$, where $a\preccurlyeq b$ if and only if $a^{-1}b\in A_S^+$.
We will denote by~$\wedge$ and~$\vee$ the meet and join operations, respectively, in this lattice.
The (minimal) \emph{Garside element} of~$A_S$ is $\Delta_S=s_1\vee\cdots \vee s_n$, where $S=\{s_1,\ldots,s_n\}$.
One can similarly define the suffix order~$\succcurlyeq$, where $b\succcurlyeq a$ if and only if $ba^{-1}\in A_S^+$.
In general, $a\preccurlyeq b$ is not equivalent to $b\succcurlyeq a$.

If $A_S$ is irreducible (that is, if the defining Coxeter graph $\Gamma_S$ is connected), then either~$\Delta_S$ or~$\Delta_S^2$ generates the center of~$A_S$ \cite{BS,Del}.
Actually, conjugation by~$\Delta_S$ induces a permutation of~$S$. We will denote $\tau_S(x)=\Delta_S^{-1}x\Delta_S$ for every element~$x\in A_S$; notice that the automorphism~$\tau_S$ has either order~$1$ or order~$2$.
The triple $(A_S, A_S^+,\Delta_S)$ determines the so called {\it classical Garside structure} of the group~$A_S$.

The \emph{simple elements} are the positive prefixes of~$\Delta_S$, which coincide with the positive suffixes of~$\Delta_S$. In an Artin--Tits group~$A_S$ of spherical type, a simple element is a positive element that is \emph{square-free}, that is, that cannot be written as a positive word with two consecutive equal letters \cite{BS,Del}.
For a simple element~$s$, we define its \emph{right complement} as $\partial_S(s)=s^{-1}\Delta_S$.
Notice that $\partial_S$ is a bijection of the set of simple elements. Notice also that $\partial_S^2(s)=\tau_S(s)$ for every simple element~$s$.

The \emph{left normal form} of an element $\alpha \in A_S$ is the unique decomposition of the form $\alpha=\Delta_S^p \alpha_1\cdots \alpha_r$, where $p\in \mathbb Z$, $r\geq 0$, every~$\alpha_i$ is a non-trivial simple element different from~$\Delta_S$, and $\partial_S(\alpha_i)\wedge \alpha_{i+1}=1$ for $i=1,\ldots,r-1$.

The numbers~$p$ and~$r$ are called the \emph{infimum} and the \emph{canonical length} of~$\alpha$, denoted $\inf_S(\alpha)$ and $\ell_S(\alpha)$, respectively.
The \emph{supremum} of~$\alpha$ is $\sup_S(\alpha)=\inf_S(\alpha)+\ell_S(\alpha)$.
By~\cite{ElM}, we know that~$\inf_S(\alpha)$ and~$\sup_S(\alpha)$ are, respectively, the maximum and minimum integers~$p$ and~$q$ such that
$\Delta_S^p\preccurlyeq \alpha \preccurlyeq \Delta_S^q$, or equivalently $\Delta_S^q\succcurlyeq \alpha \succcurlyeq \Delta_S^p$, holds.

There is another normal form which will be important for us.
The \emph{mixed normal form}, also called the {\it negative-positive} normal form, or $np$-normal form, is the decomposition of an element~$\alpha$ as $\alpha=x_s^{-1}\cdots x_1^{-1} y_1\cdots y_t$, where $x=x_1\cdots x_s$ and $y=y_1\cdots y_t$ are positive elements written in left normal form (here some of the initial factors of either~$x$ or~$y$ can be equal to~$\Delta_S$), such that $x\wedge y=1$ (that is, there is no possible cancellation in~$x^{-1}y$).

This decomposition is unique, and it is closely related to the left normal form of~$\alpha$.
Indeed, if $x\neq 1$ and $y\neq 1$ then $\inf_S(x)=\inf_S(y)=0$ holds; otherwise there would be cancellations.
In this case, if one writes $x_i^{-1}=\partial_S(x_i)\Delta_S^{-1}$ for $i=1,\ldots,s$, and then collects all appearances of~$\Delta_S^{-1}$ on the left, one gets $\alpha=\Delta_S^{-s}\widetilde x_s\cdots \widetilde x_1 y_1\cdots y_t$, where $\widetilde x_i= \tau^{-i}(\partial_S(x_i))$.
The latter is precisely the left normal form of~$\alpha$.
If~$x$ is trivial, then $\alpha= y_1\cdots y_t$ where the first~$p$ factors could be equal to~$\Delta_S$, so the left normal form is $\alpha=\Delta_S^p y_{p+1}\cdots y_t$.
If $y$ is trivial, then $\alpha=x_s^{-1}\cdots x_1^{-1}$ where some (say~$k$) of the rightmost factors could be equal to~$\Delta_S^{-1}$.
The left normal form of~$\alpha$ in this case would be $\alpha=\Delta_S^{-s}\widetilde x_s\cdots \widetilde x_{k+1}$.

Notice that if $x\neq 1$ then $\inf_S(\alpha)=-s$, and if $y\neq 1$ then $\sup_S(\alpha)=t$.

The $np$-normal form can be computed from any decomposition of~$\alpha$ as $\alpha=\beta^{-1}\gamma$, where~$\beta$ and~$\gamma$ are positive elements:
One just needs to cancel $\delta=\beta\wedge \gamma$ in the middle.
That is, write $\beta =\delta x$ and $\gamma=\delta y$; then $\alpha= \beta^{-1}\gamma = x^{-1}\delta^{-1}\delta y = x^{-1}y$, where no more cancellation is possible.
Then compute the left normal forms of~$x$ and~$y$ and the~$np$-normal form will be computed.

The mixed normal form is very useful for detecting whether an element belongs to a proper standard parabolic subgroup.
If $\alpha\in A_X$, with $X\subseteq S$, and \hbox{$\alpha=x_s^{-1}\cdots x_1^{-1} y_1\cdots y_t$} is the~$np$-normal form of~$\alpha$ in~$A_X$, then the simple factors $x_1,\ldots,x_s,y_1,\ldots,y_t \in A_X$ are also simple elements in $A_S$ such that $x_1\cdots x_s$ and $y_1\cdots y_t$ are in left normal form in~$A_S$.
Hence, the above is also the~$np$-normal form of~$\alpha$ in~$A_S$.
Therefore, given~$\alpha\in A_S$ we will have $\alpha\in A_X$ if and only if all factors in its~$np$-normal form belong to~$A_X$.
It follows that if $x\neq 1$, $\inf_S(\alpha)=\inf_X(\alpha)=-s$, and if $y\neq 1$, $\sup_S(\alpha)=\sup_X(\alpha)=t$.
\medskip

We finish this section with an important observation:
The Artin--Tits group~$A_S$ admits other Garside structures $(A_S,A_S^+,\Delta_S^N)$, which are obtained by replacing the Garside element~$\Delta_S$ with some non-trivial positive power~$\Delta_S^N$.
To see this, recall that $\Delta_S^p\preccurlyeq \alpha \preccurlyeq \Delta_S^q$ is equivalent to $\Delta_S^q\succcurlyeq \alpha \succcurlyeq \Delta_S^p$ for any $p,q\in\Z$, so the positive prefixes of~$\Delta_S^N$ coincide with the positive suffixes of~$\Delta_S^N$,
and note that one has $s\preccurlyeq \Delta_S \preccurlyeq \Delta_S^N$ for any $s\in S$, so the divisors of~$\Delta_S^N$ generate~$A_S$.

The simple elements with respect to this Garside structure are the positive prefixes of~$\Delta_S^N$ (which are no longer square-free, in general, if $N>1$).
The $np$-normal form of an element with respect to the Garside structure $(A_S,A_S^+,\Delta_S^N)$ can be obtained from that with respect to the Garside structure $(A_S,A_S^+,\Delta_S)$ by grouping together the simple factors of the positive and the negative parts of the latter in groups of~$N$, ``padding" the outermost groups with copies of the identity element as necessary.

Suppose that $\alpha=x_s^{-1}\cdots x_1^{-1} y_1\cdots y_t$ is in $np$-normal form, for the classical Garside structure of~$A_S$.
If we take $N\geq \max(s,t)$, it follows that $x=x_1\cdots x_s\preccurlyeq \Delta_S^N$ and $y=y_1\cdots y_t\preccurlyeq \Delta_S^N$.
This means that~$x$ and~$y$ are simple elements with respect to the Garside structure in which~$\Delta_S^N$ is the Garside element.
Therefore, for every $\alpha\in A_S$, we can consider a Garside structure of~$A_S$ such that the~$np$-normal form of~$\alpha$ is~$x^{-1}y$, with~$x$ and~$y$ being simple elements.

\section{Other results for Artin--Tits groups of spherical type}

This section focuses on some further properties of Artin--Tits groups of spherical type that we will need.
The properties listed in this section are specific to Artin--Tits groups and do not directly extend to Garside groups in general.

\begin{definition}
Let~$A_S$ be an Artin--Tits group of spherical type.
Given $X\subseteq S$ and~$t\in S$, we define the positive element
$$
  r_{X,t}=\Delta_{X}^{-1}\Delta_{X\cup\{t\}}.
$$
\end{definition}

If $t\notin X$, this positive element coincides with the {\it elementary ribbon}~$d_{X,t}$ defined in~\cite{God}.
If $t\in X$ we just have $r_{X,t}=1$ while $d_{X,t}=\Delta_X$.

\begin{lemma}\label{L_ribbon_permutation}
If~$A_S$ is an Artin--Tits group of spherical type, $X\subseteq S$ and $t\in S$,
then there is a subset~$Y\subseteq X\cup\{t\}$ such that $r_{X,t}^{-1}Xr_{X,t}=Y$ holds.
\end{lemma}
\begin{proof}
This follows from the definition, as the automorphisms~$\tau_X$ and~$\tau_{X\cup\{t\}}$ permute the elements of~$X$ and $X\cup\{t\}$, respectively.
\end{proof}

\begin{lemma}\label{L:ribbon_lcm}
If~$A_S$ is an Artin--Tits group of spherical type, $X\subseteq S$ and $t\in S$,
then the element $r_{X,t}$ can be characterized by the following property:
$$
\Delta_X \vee t= \Delta_X r_{X,t}
$$
\end{lemma}

\begin{proof}
This follows immediately from the definition of~$\Delta_X$, which is the least common multiple of the elements in~$X$, and the definition of $\Delta_{X\cup \{t\}}= \Delta_X r_{X,t}$, which is the least common multiple of the elements in~$X\cup \{t\}$.
\end{proof}

\begin{lemma}\label{L:ribbon_prefix}
If~$A_S$ is an Artin--Tits group of spherical type, $X\subsetneq S$, and $t\in S\setminus X$,
then the following hold:
 \begin{enumerate}[\quad\upshape (a)]
  \item $t\preccurlyeq r_{X,t}$
  \item If $s\in S$ with $s\preccurlyeq r_{X,t}$, then $s=t$.
 \end{enumerate}
\end{lemma}
\begin{proof}
We start by proving (b). Recall that the set of prefixes of a Garside element with the classical Garside structure coincides with its set of suffixes.
If $s\in S$ with $s\preccurlyeq r_{X,t}$, then one has $s\preccurlyeq r_{X,t}\preccurlyeq\Delta_{X\cup\{t\}}$,
 so $s\in X\cup\{t\}$.
 On the other hand, $\Delta_X r_{X,t}$ is a simple element by \autoref{L:ribbon_lcm}, hence square-free, and $\Delta_X\succcurlyeq u$ for all $u\in X$, so $s\notin X$ and part~(b) is shown.
Turning to the proof of (a), as $t\notin X$, one has $t\not\preccurlyeq\Delta_X$ and thus $r_{X,t}\neq 1$ by \autoref{L:ribbon_lcm}. This means that~$r_{X,t}$ must start with some letter, which by part~(b) must necessarily be~$t$.
\end{proof}

We define the \emph{support} of a positive element~$\alpha\in A_S$, denoted by~$\mbox{supp}(\alpha)$, as the set of generators~$s\in S$ that appear in any positive word representing~$\alpha$.
This is well defined by two reasons:
Firstly, two positive words represent the same element in~$A_S$ if and only if one can transform the former into the latter by repeatedly applying the relations in the presentation of~$A_S$, that is, if and only if the elements are the same in the monoid~$A_S^+$ defined by the same presentation as~$A_S$ \cite{Par}.
Secondly, due to the form of the relations in the presentation of~$A_S^+$, applying a relation to a word does not modify the set of generators that appear, so all words representing an element of~$A_S^+$ involve the same set of generators.

For a not necessarily positive element $\alpha\in A_S$, we define its {\em support}, using the $np$-normal form $\alpha=x^{-1}y=x_s^{-1}\cdots x_1^{-1} y_1\cdots y_t$, as $\mbox{supp}(\alpha)=\mbox{supp}(x)\cup \mbox{supp}(y)$.

\begin{lemma}\label{L:simple_plus_letter}
Let~$\alpha$ be a simple element (with respect to the usual Garside structure) in an Artin--Tits group~$A_S$ of spherical type.
Let~$t, s\in S$. Then:
$$
\left. \begin{array}{l}
  t\not\preccurlyeq \alpha \\
  t\preccurlyeq \alpha s
\end{array}
\right\} \Rightarrow  \alpha s = t\alpha
$$
\end{lemma}

\begin{proof}
This result follows immediately from~\cite[Lemma~1]{Dig}, but we will nevertheless provide a proof.
As the relations defining~$A_S^+$ are homogeneous, the word length of~$\alpha$ is well defined.
We proceed by induction on the length of~$\alpha$.
If $\alpha=1$ the result is trivially true, so suppose $\alpha\neq 1$ and that the result is true for shorter elements.

Let $a\in S$ such that $a\preccurlyeq \alpha$. There is a relation in~$A_S$ of the form $atat\cdots = tata\cdots$, where the words on each side have length $m=m(a,t)$.
Let us denote by~$\rho_i$ the~$i$-th letter of the word $atat\cdots$, for $i=1,\ldots,m$.
That is, $\rho_i=a$ if~$i$ is odd, and $\rho_i=t$ if~$i$ is even.
Recall that $\rho_1\cdots \rho_m= a\vee t$.
Also, $\rho_1\cdots \rho_m = t \rho_1\cdots \rho_{m-1}$.

We have $t\preccurlyeq \alpha s$ and $a\preccurlyeq \alpha \preccurlyeq \alpha s$, so $a\vee t =\rho_1 \cdots \rho_m \preccurlyeq \alpha s$.
Notice that $\rho_1\cdots \rho_m = a\vee t\not\preccurlyeq \alpha$ (as $t\not\preccurlyeq \alpha$), but $\rho_1=a\preccurlyeq \alpha$.
Hence, there is some~$k$, where $0<k<m$, such that $\rho_1\cdots \rho_k \preccurlyeq \alpha$ and $\rho_1\cdots \rho_{k+1}\not \preccurlyeq \alpha$.

Write $\alpha=\rho_1\cdots \rho_k \alpha_0$. Then $\rho_{k+1}\not\preccurlyeq \alpha_0$, but $\rho_{k+1}\preccurlyeq \alpha_0 s$. By induction hypothesis, $\alpha_0 s = \rho_{k+1} \alpha_0$.
Hence $\alpha s = \rho_1\cdots \rho_{k+1} \alpha_0$, where $k+1\leq m$.

We claim that $k+1=m$.
Otherwise, as $\rho_1\cdots \rho_m \preccurlyeq \alpha s = \rho_1\cdots \rho_{k+1} \alpha_0$, we would have $\rho_{k+2} \preccurlyeq \alpha_0$, and then $\rho_1\cdots \rho_k \rho_{k+2} \preccurlyeq \alpha$, which is not possible as  $\rho_k=\rho_{k+2}$ and~$\alpha$ is simple (thus square-free).

Hence $k+1=m$ and $\alpha s=\rho_1\cdots \rho_m \alpha_0 = t \rho_1 \cdots \rho_{m-1} \alpha_0 = t \alpha$.
\end{proof}

\begin{remark}
The above result is not true if~$\alpha$ is not simple (with the usual Garside structure).
As an example, consider the Artin--Tits group $\langle a,b \mid abab=baba \rangle$, and the elements $\alpha=aaba$ and $t=s=b$. We have $b\not\preccurlyeq \alpha$, but
$$
    \alpha b = aabab = ababa = babaa,
$$
so $b\preccurlyeq \alpha b$, but $\alpha b\neq b\alpha$ as $\alpha=aaba\neq abaa$.
\end{remark}

We end this section with an important property concerning the central elements~$z_P$.

\begin{lemma}~\label{L:commuting cP_cQ}
Let~$P$ and~$Q$ be parabolic subgroups of an Artin--Tits group~$A_S$ of spherical type.
Then the following are equivalent:
\begin{enumerate}[\quad\upshape 1.]
 \item
   $z_P z_Q = z_Q z_P$.
 \item
   $(z_P)^m (z_Q)^n = (z_Q)^n (z_P)^m$ for some $n,m\neq 0$.
 \item
   $(z_P)^m (z_Q)^n = (z_Q)^n (z_P)^m$ for all $n,m\neq 0$.
\end{enumerate}
\end{lemma}

\begin{proof}
If~$z_P$ and~$z_Q$ commute, it is clear that~$(z_P)^m$ and~$(z_Q)^n$ commute for every $n,m\neq 0$.
Conversely, suppose that~$(z_P)^m$ and~$(z_Q)^n$ commute for some $n,m\neq 0$.

By Godelle~\cite[Proposition 2.1]{God}, if $X,Y\subseteq S$ and $m\neq 0$, then $u^{-1} (z_X)^m u \in A_Y$ holds if and only if $u=vy$ with $y\in A_Y$ and $v^{-1}Xv=Y$.

Hence, if $u^{-1}(z_X)^m u = (z_X)^m$ we can take $Y=X$, so~$y$ commutes with~$z_X$ and~$v$ induces a permutation of~$X$, which implies that $v^{-1}A_Xv=A_X$ and then $v^{-1}z_X v = z_X$.
Therefore $u^{-1}z_X u = y^{-1}v^{-1}(z_X) vy = y^{-1} z_X y = z_X$.

Now recall that $(z_P)^m (z_Q)^n = (z_Q)^n (z_P)^m$.
Since $\alpha^{-1} P \alpha =A_X$ for some $\alpha\in A_S$ and some $X\subseteq S$, we can conjugate the above equality by~$\alpha$ to obtain $(z_X)^m (\alpha^{-1} z_Q \alpha)^n = (\alpha^{-1} z_Q \alpha)^n (z_X)^m$, which by the argument in the previous paragraph implies that $z_X (\alpha^{-1} z_Q \alpha)^n = (\alpha^{-1} z_Q \alpha)^n z_X$.
Conjugating back, we get $z_P (z_Q)^n = (z_Q)^n z_P$.

Now take $\beta\in A_S$ such that $\beta^{-1}Q \beta = A_Y$ for some $Y\subseteq S$.
Conjugating the equality from the previous paragraph by~$\beta$, we obtain $(\beta^{-1} z_P \beta) (z_Y)^n = (z_Y)^n(\beta^{-1} z_P \beta)$, which implies $(\beta^{-1} z_P \beta) z_Y = z_Y (\beta^{-1} z_P \beta)$.
Conjugating back, we finally obtain $z_Pz_Q=z_Qz_P$.
\end{proof}

\section{Cyclings, twisted cyclings and summit sets}\label{S:CyclingsAndFriends}

Let $(A_S, A_S^+, \Delta)$ be any Garside structure for~$A_S$, and let $\alpha\in A_S$ be an element whose left normal form is $\alpha=\Delta^p \alpha_1\cdots \alpha_r$ with~$r>0$.

The \emph{initial factor} of~$\alpha$ is the simple element $\iota(\alpha)= \tau^{-p}(\alpha_1)$.
Thus, $\alpha=\iota(\alpha)\Delta^p \alpha_2\cdots \alpha_r$.
The \emph{cycling} of~$\alpha$ is the conjugate of~$\alpha$ by its initial factor, that is,
$$
    c(\alpha)=\Delta^{p} \alpha_2\cdots \alpha_r \iota(\alpha).
$$
This expression is not necessarily in left normal form, so in order to apply a new cycling, one must first compute the left normal form of~$c(\alpha)$ to know the new conjugating element $\iota(c(\alpha))$.
If $r=0$, that is, if $\alpha=\Delta^p$, we just define $c(\Delta^p)=\Delta^p$.

The \emph{twisted cycling} of~$\alpha$ is defined as $\widetilde c(\alpha)=\tau^{-1}(c(\alpha))$. It is the conjugate of~$\alpha$ by $\iota(\alpha)\Delta^{-1}$, which is the {\it inverse} of a simple element.
(One can also think of~$\widetilde c$ as a left-conjugation by a simple element.)
Notice that the conjugating element is
$$
\iota(\alpha)\Delta^{-1} = \Delta^{p}\alpha_1 \Delta^{-(p+1)}.
$$

The following lemma tells us that twisted cycling is actually more natural than cycling from the point of view of the mixed normal form.

\begin{lemma}~\label{L:twisted_cycling}
If $\alpha=x^{-1}y=x_s^{-1}\cdots x_1^{-1} y_1\cdots y_t$ is the $np$-normal form of an element in~$A_S$, and $x\neq 1$, then the conjugating element for twisted cycling is precisely~$x_s^{-1}$.
Hence,
$$
\widetilde c(\alpha) = x_{s-1}^{-1}\cdots x_1^{-1} y_1\cdots y_t x_s^{-1}.
$$
\end{lemma}

\begin{proof}
We have seen that the conjugating element for the twisted cycling of~$\alpha$ is $\Delta^{-s}\alpha_1 \Delta^{s-1}$, where~$\alpha_1$ is the first non-$\Delta$ factor in its left normal form, that is, $\alpha_1= \widetilde x_s = \tau^{-s}(\partial(x_s)) = \Delta^s \partial(x_s) \Delta^{-s}$.

Thus, the conjugating element is $\Delta^{-s} \Delta^s \partial(x_s) \Delta^{-s} \Delta^{s-1} = \partial(x_s)\Delta^{-1} = x_s^{-1}$.
\end{proof}

If $\alpha=\Delta^p \alpha_1\cdots \alpha_r$ is in left normal form and $r>0$, the \emph{decycling} of~$\alpha$ is the conjugate of~$\alpha$ by~$\alpha_r^{-1}$, that is,
$$
    d(\alpha)=\alpha_r\Delta^p \alpha_1\cdots \alpha_{r-1}.
$$
If $r=0$, that is, if $\alpha=\Delta^p$, we define $d(\Delta^p)=\Delta^p$.
Cyclings (or twisted cyclings) and decyclings are used to compute some important finite subsets of the conjugacy class of an element, which we will refer to with the common name of \emph{summit sets}.

\begin{definition}
Given $\alpha\in A_S$, we denote by~$C^+(\alpha)$ the set of positive conjugates of~$\alpha$.
Notice that this set is always finite, and it could be empty.
\end{definition}

\begin{definition}{\rm\cite{ElM}}
Given $\alpha \in A_S$, the {\em Super Summit Set} of~$\alpha$, denoted by~$SSS(\alpha)$, is the set of all conjugates of~$\alpha$ with maximal infimum and minimal supremum.
Equivalently,~$SSS(\alpha)$ is the set of conjugates of~$\alpha$ with minimal canonical length.
\end{definition}

\begin{definition}{\rm\cite{Geb}}
Given $\alpha\in A_S$, the {\em Ultra Summit Set} of~$\alpha$, denoted by~$USS(\alpha)$, is the set of all elements $\beta\in SSS(\alpha)$ such that $c^{k}(\beta)=\beta$ for some $k>0$.
\end{definition}

It is easy to deduce from the definitions that~$c$ and~$\tau$ commute; indeed, for any $\beta\in A_S$, the conjugating elements for the conjugations $\beta\mapsto c(\tau(\beta))$ and $\beta\mapsto \tau(c(\beta))$ are identical.
Therefore, $c^{k}(\beta)=\beta$ for some $k>0$ holds if and only if $\widetilde c^{t}(\beta)=\beta$ for some $t>0$.
This means that twisted cycling can be used to define the Ultra Summit Set.

We will also need two other types of summit sets:

\begin{definition}{\rm\cite{Lee}}
Given~$\alpha\in A_S$, the {\em Reduced Super Summit Set} of~$\alpha$ is
$$
  RSSS(\alpha)= \{x\in \alpha^G : c^{k}(x)=x=d^t(x) \mbox{ for some } k,t>0\}.
$$
\end{definition}

As above, one can use twisted cycling instead of cycling to define $RSSS(\alpha)$.

\begin{definition}{\rm\cite{BGG}}
Given~$\alpha\in A_S$, the  {\em stable ultra summit set} of~$\alpha$ is $ SU(\alpha)= \{x\in \alpha^G : x^m\in USS(x^m)\; \forall m\in \mathbb Z\}$.
\end{definition}

It is well known~\cite{ElM,Geb} that when applying iterated cycling, starting with $\alpha\in A_S$, one obtains an element~$\alpha'$ whose infimum is maximal in its conjugacy class.
Then, by applying iterated decycling to~$\alpha'$ one obtains an element~$\alpha''$ whose supremum is minimal in its conjugacy class, so $\alpha''\in SSS(\alpha)$.
Finally, when applying iterated cycling to~$\alpha''$ until the first repeated element~$\alpha'''$ is obtained, one has $\alpha'''\in USS(\alpha)$.
If one then applies iterated decycling to~$\alpha'''$ until the first repeated element, one obtains $\widehat\alpha \in RSSS(\alpha)$.

In order to conjugate~$\widehat\alpha$ to~$SU(\alpha)$, as explained in~\cite{BGG}, one just needs to apply the conjugating elements for iterated cycling or decycling of suitable powers of~$\widehat \alpha$. It follows that all the above summit sets are nonempty (and finite), and moreover we have the following:

\begin{lemma}\label{L:conjugating_elements_to_summit_sets}
Let~$\alpha\in A_S$, and let~$I$ be either $SSS(\alpha)$, or $USS(\alpha)$, or $RSSS(\alpha)$, or $SU(\alpha)$. One can conjugate~$\alpha$ to $\beta\in I$ by a sequence of conjugations:
$$
   \alpha=\alpha_0 \rightarrow \alpha_1\rightarrow \cdots \rightarrow \alpha_m=\beta
$$
where, for $i=0,\ldots,m-1$, the conjugating element from $\alpha_{i}$ to $\alpha_{i+1}$ has the form $\alpha_i^p\Delta^q\wedge \Delta^r$ for some integers $p$, $q$ and $r$.
\end{lemma}

\begin{proof}
We just need to show that the conjugating elements, for either cycling or decycling, of a power~$x^k$ of an element~$x$ have the form $x^p\Delta^q\wedge \Delta^r$ for some integers~$p$, $q$ and~$r$.

If~$x^k$ is a power of~$\Delta$, the conjugating element is trivial, so the result holds. Otherwise, suppose that $\Delta^m x_1\cdots x_n$ is the left normal form of~$x^k$. Then $x^k \Delta^{-m} = \tau^{-m}(x_1)\cdots \tau^{-m}(x_n)$, where the latter decomposition is in left normal form. Hence, $x^k \Delta^{-m} \wedge \Delta = \tau^{-m}(x_1) = \iota(x^k)$. So the conjugating element for cycling has the desired form.

On the other hand, $\Delta^{m+n-1} \wedge x^k = \Delta^m x_1\cdots x_{n-1}$. So $x_n^{-1} = x^{-k} (\Delta^{m+n-1} \wedge x^k) = x^{-k}\Delta^{m+n-1}\wedge 1$. Since $x_n^{-1}$ is the conjugating element for decycling of~$x^k$, the result follows.
\end{proof}

The sets~$C^+(\alpha)$,~$SSS(\alpha)$,~$USS(\alpha)$,~$RSSS(\alpha)$ and ~$SU(\alpha)$ share a common important property, which is usually called {\it convexity}. Recall that~$\beta^x$ means $x^{-1}\beta x$ and that~$\wedge$ denotes the meet operation in the lattice associated to the prefix partial order~$\preccurlyeq$.

\begin{lemma}\label{L:convexity}{\rm\cite[Propositions 4.8 and 4.12]{FGM}, \cite[Theorem 1.18]{Geb}}
Let~$\alpha\in A_S$, and let~$I$ be either $C^+(\alpha)$, or $SSS(\alpha)$, or $USS(\alpha)$, or~$RSSS(\alpha)$, or~$SU(\alpha)$.
If $\alpha, \alpha^x, \alpha^y \in I$ then $\alpha^{x\wedge y}\in I$.
\end{lemma}

\begin{remark}
Convexity of $SU(\alpha)$ is not shown in~\cite{BGG}, but it follows immediately from convexity of~$USS(\alpha)$.
\end{remark}

\begin{remark}
Convexity of $RSSS(\alpha)$ also follows from convexity of $USS(\alpha)$, after the following observation: An element $x$ belongs to a closed orbit under decycling if and only if~$x^{-1}$ belongs to a closed orbit under twisted cycling, hence to a closed orbit under cycling. Actually, the conjugating element for twisted cycling of~$x^{-1}$ equals the conjugating element for decycling of~$x$. Hence the elements in $RSSS(\alpha)$ are those elements~$x$, conjugate to~$\alpha$, such that both~$x$ and~$x^{-1}$ belong to closed orbits under cycling. This implies the convexity of~$RSSS(\alpha)$. Note that it also implies that $SU(\alpha) \subseteq RSSS(\alpha)$.
\end{remark}

The set~$I$ is usually obtained by computing the directed graph~$\mathcal G_I$, whose vertices are the elements of~$I$, and whose arrows correspond to \emph{minimal positive conjugators}.
That is, there is an arrow labelled~$x$ starting from a vertex~$u$ and finishing at a vertex~$v$ if and only if:
\begin{enumerate}
\item $x\in A_S^+$;
\item $u^x=v$; and
\item $u^y\notin I$ for every non-trivial proper prefix~$y$ of~$x$.
\end{enumerate}

Thanks to the convexity property, one can see that the graph~$\mathcal G_I$ is finite and connected, and that the label of each arrow is a simple element.
This is why that graph can be computed starting with a single vertex, iteratively conjugating the known vertices by simple elements, until no new elements of~$I$ are obtained.
\medskip

We will need to work with different Garside structures $(A_S,A_S^+,\Delta_S^n)$ for distinct positive values of~$n$ simultaneously.
To distinguish the precise Garside structure we are using, we will write~$I_n$ instead of~$I$ for a given summit set. For instance, given $n\geq 1$ and $\alpha\in A_S$, we will write~$SSS_n(\alpha)$,~$USS_n(\alpha)$,~$RSSS_n(\alpha)$ and~$SU_n(\alpha)$ to refer to the super summit set, the ultra summit set, the reduced super summit set and the stable ultra summit set, respectively, of~$\alpha$ with respect to the Garside structure $(A_S,A_S^+,\Delta_S^n)$. Notice that all those sets are finite sets consisting of conjugates of~$\alpha$. Also, the set~$C^+(\alpha)$ is independent of the Garside structure under consideration, so $C_n^+(\alpha)=C^+(\alpha)$ for every $n>0$.

We will now see that for every type of summit set~$I$, there is always some element which belongs to~$I_n$ for every $n>1$.

\begin{definition}
Let $\alpha\in A_S$, and let~$I$ be either $SSS$, or $USS$, or~$RSSS$, or~$SU$.
Then we define $I_\infty(\alpha)=\bigcap_{n\geq 1}{I_n}(\alpha)$.
\end{definition}

\begin{proposition}\label{P:I_infty_is_nonempty}
For every $\alpha\in A_S$, the set $I_\infty(\alpha)$ is nonempty.
\end{proposition}

\begin{proof}
We write~$I_n$ for~$I_n(\alpha)$ for $n\in\N\cup\{\infty\}$.
For every $N\geq 1$, let $I_{\leq N}=\bigcap_{n=1}^{N}{I_n}$. We will show that $I_{\leq N}\neq\emptyset$ by induction on~$N$.

If $N=1$ then $I_{\leq N}=I_1$, which is known to be nonempty. We can then assume that $N>1$ and that there is an element $\beta \in I_{\leq (N-1)}$. Using the Garside structure $(A_S,A_S^+,\Delta_S^N)$, we can conjugate~$\beta$ to an element $\gamma\in I_N$ by applying some suitable conjugations
$$
  \beta=\beta_0 \rightarrow \beta_1 \rightarrow \cdots \rightarrow \beta_m = \gamma.
$$
By~\autoref{L:conjugating_elements_to_summit_sets}, for every $i=0,\ldots,m-1$ the conjugating element from~$\beta_i$ to~$\beta_{i+1}$ has the form $\beta_i^p \Delta_S^{Nq} \wedge \Delta_S^{Nr}$ for some integers $p$, $q$ and~$r$.  Recall that $\beta\in I_{\leq (N-1)}$. We claim that, if we apply such a conjugating element to~$\beta$, the resulting element $\beta_1$ will still belong to $I_{\leq (N-1)}$. Repeating the argument $m$~times, it will follow that $\gamma\in I_{\leq (N-1)}\cap I_N = I_{\leq N}$, so $I_{\leq N}\neq \emptyset$ as we wanted to show.

Notice that, if $\beta \in I_n$ for some~$n$, then $\beta^{\beta^p \Delta_S^{Nq}} =\beta^{\Delta_S^{Nq}}= \tau_S^{Nq}(\beta)\in I_n$, and also $\beta^{\Delta_S^{Nr}}=\tau_S^{Nr}(\beta)\in I_n$. By~\autoref{L:convexity}, it follows that $\beta_1\in I_n$. Hence, if $\beta\in I_{\leq (N-1)}$ then $\beta_1\in I_{\leq (N-1)}$, which shows the claim.

Therefore, we have shown that $I_{\leq N}\neq \emptyset$ for every $N\geq 1$. As the set~$I_1$ is finite, and we have a monotonic chain
$$
   I_1 = I_{\leq 1} \supseteq I_{\leq 2} \supseteq I_{\leq 3} \supseteq \cdots,
$$
this chain must stabilize for some $N\geq 1$, so $I_{\infty}=I_{\leq N}\neq \emptyset$.
\end{proof}
\medskip

We finish this section by recalling an important tool for studying properties of the Ultra Summit Set: the \emph{transport map}.
Let again $(A_S, A_S^+, \Delta)$ be any Garside structure for~$A_S$.

\begin{definition}{\rm\cite{Geb}}
Let $\alpha\in A_S$ and $v,w\in USS(\alpha)$.
Let $x\in A_S$ be an element conjugating~$v$ to~$w$, that is, $x^{-1}vx=w$.
For every $i\geq 0$, let $v^{(i)}=c^i(v)$ and $w^{(i)}=c^i(w)$.
The {\em transport} of~$x$ at~$v$ is the element $x^{(1)}=\iota(v)^{-1}x\:\iota(w)$, which conjugates $v^{(1)}$ to $w^{(1)}$.

We can define iteratively~$x^{(i)}$ as the transport of~$x^{(i-1)}$ at~$v^{(i-1)}$.
That is, the {\em $i$-th transport} of~$x$ at~$v$, denoted~$x^{(i)}$, is the following conjugating element from~$v^{(i)}$ to~$w^{(i)}$:
$$
  x^{(i)}= \left(\iota(v)\iota(v^{(1)})\cdots \iota(v^{(i-1)})\right)^{-1}x\: \left(\iota(w)\iota(w^{(1)})\cdots \iota(w^{(i-1)})\right).
$$
\end{definition}

\begin{lemma}\label{L:transport}{\rm\cite[Lemma 2.6]{Geb}}
Let $\alpha\in A_S$ and $v,w\in USS(\alpha)$.
For every conjugating element~$x$ such that $x^{-1}vx=w$ there exists some integer $N>0$ such that $v^{(N)}=v$, $w^{(N)}=w$ and $x^{(N)}=x$.
\end{lemma}

We remark that in~\cite{Geb} it is assumed that~$x$ is a positive element, but this is not a constraint:
Multiplying~$x$ by a suitable central power of the Garside element~$\Delta$, we can always assume that~$x$ is positive.

\autoref{L:transport} can be rewritten in the following way:

\begin{lemma}\label{L:conjugated_conjugating_elements}
Let $\alpha\in A_S$ and $v,w\in USS(\alpha)$.
Let~$m$ and~$n$ be the lengths of the orbits under cycling of~$v$ and~$w$, respectively. Denote by $C_t(v)$ (resp.\ $C_t(w)$) the product of~$t$ consecutive conjugating elements for cycling, starting with~$v$ (resp.\ starting with~$w$). Then, for every~$x$ such that $x^{-1}v x=w$, there is a positive common multiple~$N$ of~$m$ and~$n$ such that $x^{-1} C_{kN}(v)\: x = C_{kN}(w)$ for every $k>0$.
\end{lemma}

\begin{proof}
By~\autoref{L:transport}, there is some $N>0$ such that $v^{(N)}=v$, $w^{(N)}=w$ and $x^{(N)}=x$. The first property implies that~$N$ is a multiple of~$m$. The second one, that~$N$ is a multiple of~$n$. Finally, by definition of transport, the third property just means $x^{-1} C_N(v) x = C_N(w)$.

Now notice that, as~$N$ is a multiple of the length~$m$ of the orbit of~$v$ under cycling, one has $C_{kN}(v)=C_N(v)^k$ for every $k>0$. In the same way, $C_{kN}(w)=C_N(w)^k$ for every $k>0$. Therefore $x^{-1} C_{kN}(v) x = x^{-1} C_N(v)^k x =  C_N(w)^k = C_{kN}(w)$.
\end{proof}

\begin{corollary}\label{C:conjugated_conjugating_elements_twisted}
Let $\alpha\in A_S$ and $v,w\in USS(\alpha)$.
For every $t>0$, denote by $\widetilde C_t(v)$ (resp.\ $\widetilde C_t(w)$) the product of~$t$ consecutive conjugating elements for twisted cycling, starting with~$v$ (resp.\ starting with~$w$). Then, for every~$x$ such that $x^{-1}v x=w$, there is a positive integer~$M$ such that $x^{-1} \widetilde C_M(v)\: x = \widetilde C_M(w)$, where $\widetilde C_M(v)$ commutes with~$v$ and~$\widetilde C_M(w)$ commutes with~$w$.
\end{corollary}

\begin{proof}
By definition of cycling and twisted cycling, we have $\widetilde C_t(v) = C_t(v) \Delta^{-t}$ and $\widetilde C_t(w) = C_t(w) \Delta^{-t}$ for every $t>0$.

We know from~\autoref{L:conjugated_conjugating_elements} that there is some positive integer~$N$ such that $x^{-1} C_{kN}(v)\: x = C_{kN}(w)$ for every $k>0$. If we take $k$ big enough so that $\Delta^{kN}$ is central, and we denote $M=kN$, we have:
$$
   x^{-1} \widetilde C_{M}(v)\: x= x^{-1} C_M(v)\Delta^{-M} x = x^{-1} C_M(v)\: x \Delta^{-M} = C_M(w)\Delta^{-M} = \widetilde C_M(w).
$$
Moreover, from~\autoref{L:conjugated_conjugating_elements} we know that $C_M(v)$ commutes with~$v$, hence $\widetilde C_M(v)= C_M(v)\Delta^{-M}$ also commutes with~$v$. In the same way, $\widetilde C_M(w)$ commutes with~$w$.
\end{proof}

\section{Positive conjugates of elements in a parabolic subgroup}

Suppose that an element~$\alpha$ belongs to a proper parabolic subgroup $P\subsetneq A_S$ and has a positive conjugate.
We will show in this section that then \emph{all} positive conjugates of~$\alpha$ belong to a proper \emph{standard} parabolic subgroup, determined by their corresponding supports.

Moreover, the support of a positive conjugate of~$\alpha$ will be shown to be {\it preserved by conjugation}, in the sense that if~$\alpha'$ and~$\alpha''$ are two positive conjugates of~$\alpha$, whose respective supports are~$X$ and~$Y$, then every element conjugating~$\alpha'$ to~$\alpha''$ will also conjugate~$A_X$ to~$A_Y$.

This will allow us to define a special parabolic subgroup associated to~$\alpha$, which will be the smallest parabolic subgroup (by inclusion) containing~$\alpha$.
\medskip

Given $X\subsetneq S$, we will denote by~$c_X$ and~$c_S$ the cycling operations, respectively, in the Garside groups~$A_X$ and~$A_S$.
Analogous notation will be used for the other concepts introduced in \autoref{S:CyclingsAndFriends}.
\medskip

\begin{lemma}\label{L:cycling_parabolic}
If $\alpha\in A_X\subsetneq A_S$ where $X \subsetneq S$, then $c_S(\alpha)\in A_Y$, where either $Y=X$ or $Y=\tau_S(X)$.
\end{lemma}

\begin{proof}
Let $x^{-1}y=x_s^{-1}\cdots x_1^{-1} y_1\cdots y_t$ be the $np$-normal form of~$\alpha$ in~$A_S$.
As this is precisely the $np$-normal form of~$\alpha$ in~$A_X$, one has that $\mbox{supp}(x)\subseteq X$ and $\mbox{supp}(y)\subseteq X$.

Suppose that $x=1$, so $\alpha=y_1\cdots y_t$.
It is clear that $y_1\neq \Delta_S$, as $X\subsetneq S$.
Then $c_S(\alpha) = \alpha^{y_1} = y_2\cdots y_t y_1$, hence $c_S(\alpha)\in A_X$.
Notice that if $y_1\neq \Delta_X$ or if~$\alpha$ is a power of~$\Delta_X$, then one has $c_S(\alpha)=c_X(\alpha)$, but otherwise we may have $c_S(\alpha)\neq c_X(\alpha)$.

Now suppose that $x\neq 1$.
We saw in \autoref{L:twisted_cycling} that
$$
 \widetilde c_S(\alpha)=x_{s-1}^{-1}\cdots x_1^{-1}y_1\cdots y_t x_s^{-1},
$$
hence $\widetilde c_S(\alpha)= \tau_S^{-1}(c_S(\alpha))\in A_X$, which implies that $c_S(\alpha)\in A_{\tau_S(X)}$.
Notice that in this case~$c_S(\alpha)$ is not necessarily equal to~$c_X(\alpha)$, but $\widetilde c_S(\alpha)=\widetilde c_X(\alpha)$.
\end{proof}

%
%
%
%
%

\begin{lemma}\label{L:decycling_parabolic}
If $\alpha\in A_X\subsetneq A_S$ where $X \subsetneq S$, then $d_S(\alpha)\in A_Y$, where either $Y=X$ or $Y=\tau_S(X)$.
\end{lemma}

\begin{proof}
It is obvious from the conjugating elements for decycling and twisted cycling that for every element~$\alpha$, one has $d_S(\alpha)= (\widetilde c_S(\alpha^{-1}))^{-1}$.
Now $\alpha\in A_X$, so $\alpha^{-1}\in A_X$ and then $\widetilde c_S(\alpha^{-1})\in A_Y$, where either $Y=X$ or $Y=\tau_S(X)$.
Therefore, $d_S(\alpha)= (\widetilde c_S(\alpha^{-1}))^{-1}\in A_Y$.
\end{proof}

Thanks to~\autoref{L:cycling_parabolic} and~\autoref{L:decycling_parabolic}, we see that if $\alpha\in A_X$, then we obtain $\alpha''\in A_X \cap SSS(\alpha)$, where $\alpha''$ is obtained from~$\alpha$ by iterated cycling and decycling (and possibly conjugating by~$\Delta_S$ at the end).
Notice that if~$\alpha$ is conjugate to a positive element, then~$\alpha''$ will be positive (as the infimum of~$\alpha''$ is maximal in its conjugacy class).

Let us suppose that $\alpha''\in A_X$ is positive.
We want to study the graph~$\mathcal G_{C^+(\alpha)}$.
We already know that some vertex $\alpha''\in C^+(\alpha)$ belongs to a proper standard parabolic subgroup.
Let us see that this is the case for all elements in~$C^+(\alpha)$.

\begin{proposition}~\label{P:arrows_in_positive_graph}
Let $v\in C^+(\alpha)$ with $\mbox{supp}(v)=X\subsetneq S$.
Then the label of every arrow in~$\mathcal G_{C^+(\alpha)}$ starting at~$v$ either belongs to~$A_X$, or is a letter~$t\in S$ that commutes with every letter in~$X$, or is equal to~$r_{X,t}$ for some letter~$t\in S$ adjacent to~$X$ in~$\Gamma_S$.
\end{proposition}

\begin{proof}
Let~$x$ be the label of an arrow in~$\mathcal G_{C^+(\alpha)}$ starting at~$v$.
Let $t\in S$ be such that $t\preccurlyeq x$.
We distinguish three cases.
\smallskip

\underline{Case~1}:\quad
Suppose that $t\in X$.
As we know that~$\Delta_X$ conjugates~$v$ to a positive element~$\tau_X(v)$, the convexity property implies that $\Delta_X\wedge x$ also conjugates~$v$ to a positive element.
As~$x$ is minimal with this property, $\Delta_X\wedge x$ must be either trivial or equal to~$x$.
However, $\Delta_X\wedge x$ cannot be trivial as $t\preccurlyeq \Delta_X\wedge x$.
Therefore, $\Delta_X\wedge x=x$, which is equivalent to $x\preccurlyeq \Delta_X$. Hence $x\in A_X$.
\smallskip

\underline{Case~2}:\quad
Suppose that $t\notin X$ and~$t$ is not adjacent to~$X$.
Then~$t$ commutes with all letters of~$X$, which means that $v^t=v$ is a positive element.
Hence, by minimality of~$x$, we have $x=t$.
\smallskip

\underline{Case~3}:\quad
Suppose that $t\notin X$ and~$t$ is adjacent to~$X$.
We must show $x=r_{X,t}$.

We will determine~$x$ algorithmically, starting with $t \preccurlyeq x$, and iteratively adding letters until we obtain the whole conjugating element~$x$.

Write $v=v_1\cdots v_r$, where $v_i\in X$ for $i=1,\ldots,r$, and let $c_0$ be such that $t\preccurlyeq c_0\preccurlyeq x$ and $c_0\preccurlyeq r_{X,t}$.
Consider the following diagram, in which every two paths with the same initial and final vertices represent the same element:

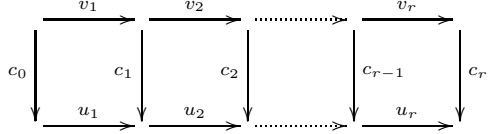
\[
\xymatrix@C=12mm@R=12mm{
  \ar[r]^{v_1} \ar[d]_{c_0}
& \ar[r]^{v_2} \ar[d]_{c_1}
& \ar@{.>}[r]  \ar[d]_{c_2}
& \ar[r]^{v_r} \ar[d]^{c_{r-1}}
& \ar[d]^{c_r}
\\
  \ar[r]^{u_1}
& \ar[r]^{u_2}
& \ar@{.>}[r]
& \ar[r]^{u_r}
&
}
\]

Starting with~$c_0$ and~$v_1$, the elements~$c_1$ and~$u_1$ are defined by the condition $c_0\vee v_1 = c_0u_1 = v_1c_1$, corresponding to the first square.
Then~$c_2$ and~$u_2$ are determined by the condition $c_1\vee v_2= c_1u_2 = v_2c_2$ corresponding to the second square, and so on.

Consider the following claims that will be proven later:

Claim~1: $c_0 \vee (v_1\cdots v_i) = (v_1\cdots v_i)c_i$ for every $1\leq i\leq r$.

Claim~2: $ t \preccurlyeq c_0\preccurlyeq c_1 \preccurlyeq c_2 \preccurlyeq \cdots \preccurlyeq c_r \preccurlyeq x$ and $c_r\preccurlyeq r_{X,t}$.

These claims give us a procedure to compute~$x$:
We start with $c_0=t=x_0$, and compute $c_1, c_2,\ldots, c_r=:x_1$.
If~$x_1$ is longer than~$x_0$, we start the process again, this time taking $c_0=x_1$, and compute $c_1,\ldots,c_r=:x_2$.
We keep going while~$x_{i}$ is longer than~$x_{i-1}$.
As all obtained elements are prefixes of~$x$, it follows that the process will stop and we have $x_{k-1}=x_k$ for some~$k$.
On the other hand, Claim~1 implies that $x_k\preccurlyeq vx_k$, so $v^{x_k}$ is a positive element. Now notice that $x_k\preccurlyeq x$. Hence, by minimality of~$x$, it follows that $x_k=x$.

At this point of the iteration, something interesting will happen.
We can substitute $c_0=x$ in the above diagram and obtain:
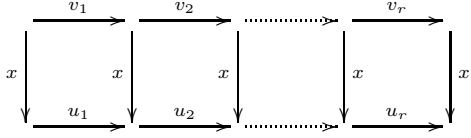
\[
\xymatrix@C=12mm@R=12mm{
  \ar[r]^{v_1} \ar[d]_{x}
& \ar[r]^{v_2} \ar[d]_{x}
& \ar@{.>}[r]  \ar[d]_{x}
& \ar[r]^{v_r} \ar[d]^{x}
& \ar[d]^{x}
\\
  \ar[r]^{u_1}
& \ar[r]^{u_2}
& \ar@{.>}[r]
& \ar[r]^{u_r}
&
}
\]

This happens because each vertical arrow is a prefix of the following one, and the first and last arrows coincide, hence all vertical arrows must be the same.
Then $v_ix=xu_i$ for every $i=1,\ldots,r$.
As the relations defining~$A_S$ are homogeneous, it follows that each~$u_i$ is a single letter, and thus~$x$ conjugates the whole set~$X$ to a set $Y\subsetneq S$ (since $\mbox{supp}(v)=X$).
But then $x^{-1}\Delta_X x$ is a positive element, that is, $x\preccurlyeq \Delta_X x$.
Hence $t\preccurlyeq x \preccurlyeq \Delta_X x$. By \autoref{L:ribbon_lcm}, $\Delta_X r_{X,t} = \Delta_X\vee t \preccurlyeq \Delta_X x$, so we have that $r_{X,t}\preccurlyeq x$.
By \autoref{L_ribbon_permutation}, $v^{r_{X,t}}$ is a positive element.
Hence, minimality of~$x$ implies $x=r_{X,t}$, as we wanted to show.

The proofs of Claim~1 and Claim~2 remain to be done:

\emph{Proof of Claim~1:}
We will show by induction on~$i$ that $c_0 \vee (v_1\cdots v_i) = (v_1\cdots v_i)c_i$.
This is true for $i=1$ by definition of~$c_1$, so assume that $i>1$ and that the claim is true for~$i-1$.
Let then $c_0 \vee (v_1\cdots v_i)=v_1\cdots v_id$.
We have
$$
   (v_1\cdots v_{i-1})c_{i-1}= c_0 \vee (v_1\cdots v_{i-1})\preccurlyeq c_0 \vee (v_1\cdots v_i) = (v_1\cdots v_i)d,
$$
and thus
$$
  (v_1\cdots v_{i-1})v_ic_i =  (v_1\cdots v_{i-1})v_i \vee (v_1\cdots v_{i-1})c_{i-1} \preccurlyeq  (v_1\cdots v_{i-1})v_i d.
$$
This implies that $c_i\preccurlyeq d$.
But $(v_1\cdots v_i)c_i$ is a common multiple of~$c_0$ and $v_1\cdots v_i$ by construction, so $d=c_i$ and the claim is shown.

\emph{Proof of Claim~2:} By hypothesis, we have that $c_0\preccurlyeq r_{X,t}$.
We will first show by induction on~$i$ that $c_i\preccurlyeq r_{X,t}$ for every $i\geq 0$. Suppose that $c_{i-1}\preccurlyeq r_{X,t}$ for some~$i$.
Since $X^{r_{X,t}}=Y$ for some $Y\subsetneq S$ by \autoref{L_ribbon_permutation}, it follows that $v_i^{r_{X,t}}$ is positive, that is, $r_{X,t}\preccurlyeq v_i r_{X,t}$.
Hence $c_{i-1}\preccurlyeq v_i r_{X,t}$ and finally $v_ic_i = c_{i-1}\vee v_i \preccurlyeq v_i r_{X,t}$, so $c_i\preccurlyeq r_{X,t}$.


Secondly, we show that $c_{i-1}\preccurlyeq c_i$ for $i\geq 1$:
Since~$v_i$ is a single letter belonging to~$X$, and $t$ is the only possible initial letter of $r_{X,t}$ by \autoref{L:ribbon_prefix}, the assumption $t\notin X$ implies that $v_i\not\preccurlyeq c_{i-1}$, hence $u_i\neq 1$.
Write the word~$c_{i-1}u_i$ as~$c_{i-1}u_i' s$, where~$s$ is a single letter.
Since $c_{i-1}u_i=c_{i-1}\vee v_i$, it follows that $v_i\not\preccurlyeq c_{i-1}u_i'$ and $v_i\preccurlyeq c_{i-1}u_i's$.
By \autoref{L:simple_plus_letter} (notice that~$c_{i-1}$ is a simple element with respect to the usual Garside structure, as it is a prefix of~$r_{X,t}$, hence~$c_{i-1}u_i'$ is simple as it is a prefix of~$c_{i-1}\vee v_i$, which is also simple), we obtain that $c_{i-1}u_i's = v_i c_{i-1}u_i'$.
That is, $v_ic_i = c_{i-1}u_i = c_{i-1}u_i' s = v_i c_{i-1} u_i'$, which implies that $c_{i-1}\preccurlyeq c_i$.

Finally, we prove that $c_r\preccurlyeq x$: As $x^{-1}vx$ is positive, we have that $c_0\preccurlyeq x \preccurlyeq vx$, and also $vc_0\preccurlyeq vx$.
Hence, $c_0\vee vc_0 \preccurlyeq vx$.
But notice that by Claim~1, $c_0\vee v= v c_r$, so $vc_r = c_0\vee v \preccurlyeq c_0\vee vc_0 \preccurlyeq vx$.
Therefore $c_r\preccurlyeq x$.
\end{proof}

\begin{example}\label{E:positive_conjugates_arrows}
In \autoref{F:positive_conjugates_arrows} we can see the graph $\mathcal G_{C^+(\alpha)}$ for $\alpha=\sigma_1\sigma_2$ in the braid group on 5~strands (the Artin--Tits group of type~$A_4$).
We see the 6~vertices corresponding to the positive conjugates of $\sigma_1\sigma_2$, and the three kinds of arrows explained in \autoref{P:arrows_in_positive_graph}.
For instance, the arrows starting from $\sigma_1\sigma_2$ are labeled~$\sigma_1$ (type 1), $\sigma_4$ (type 2) and $\sigma_3\sigma_2\sigma_1$ (type 3).
\end{example}

\begin{figure}[h]
\begin{center}
\tikzset{main node/.style={ellipse,draw,minimum size=0.53cm,inner sep=0pt},}
 \begin{tikzpicture}
    \tikzstyle{flecha}=[->, thick,>=latex]
    \node[main node] (1) {$\sigma_1\sigma_2$};
    \node[main node] (2) [right = 3cm  of 1]  {$\sigma_2\sigma_3$};
    \node[main node] (3) [right = 3cm  of 2] {$\sigma_3\sigma_4$};
    \node[main node] (4) [below = 1.5cm  of 1] {$\sigma_2\sigma_1$};
    \node[main node] (5) [right = 3cm  of 4] {$\sigma_3\sigma_2$};
    \node[main node] (6) [right = 3cm  of 5] {$\sigma_4\sigma_3$};

    \draw[flecha] (1)to[bend left] node[below]{$\sigma_3\sigma_2\sigma_1$}(2);
    \draw[flecha] (2)to[bend left] node[above]{$\sigma_1\sigma_2\sigma_3$}(1);
    \draw[flecha] (2)to[bend left] node[below]{$\sigma_4\sigma_3\sigma_2$}(3);
    \draw[flecha] (3)to[bend left] node[above]{$\sigma_2\sigma_3\sigma_4$}(2);
    \draw[flecha] (4)to[bend left] node[below]{$\sigma_3\sigma_2\sigma_1$}(5);
    \draw[flecha] (5)to[bend left] node[above]{$\sigma_1\sigma_2\sigma_3$}(4);
    \draw[flecha] (5)to[bend left] node[below]{$\sigma_4\sigma_3\sigma_2$}(6);
    \draw[flecha] (6)to[bend left] node[above]{$\sigma_2\sigma_3\sigma_4$}(5);
    \draw[flecha] (1)to[bend right] node[left]{$\sigma_1$}(4);
    \draw[flecha] (4)to[bend right] node[right]{$\sigma_2$}(1);
    \draw[flecha] (2)to[bend right] node[left]{$\sigma_2$}(5);
    \draw[flecha] (5)to[bend right] node[right]{$\sigma_3$}(2);
    \draw[flecha] (3)to[bend right] node[left]{$\sigma_3$}(6);
    \draw[flecha] (6)to[bend right] node[right]{$\sigma_4$}(3);
    \draw[flecha] (1)to[loop left] node[above left]{$\sigma_4$}(1);
    \draw[flecha] (4)to[loop left] node[below left]{$\sigma_4$}(4);
    \draw[flecha] (3)to[loop right] node[above right]{$\sigma_1$}(3);
    \draw[flecha] (6)to[loop right] node[below right]{$\sigma_1$}(6);

\end{tikzpicture}
\end{center}
\caption{The graph $\mathcal G_{C^+(\alpha)}$ for $\alpha=\sigma_1\sigma_2$ in the braid group on 5 strands.}\label{F:positive_conjugates_arrows}
\end{figure}
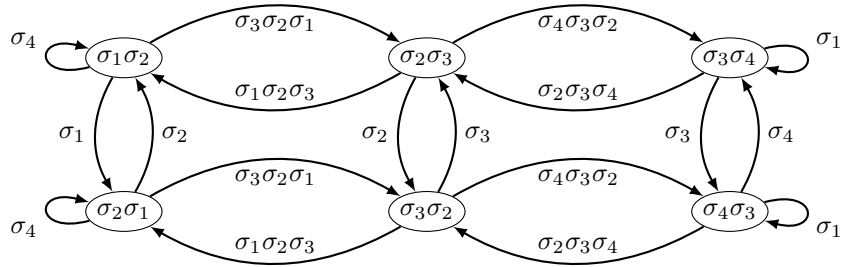

\begin{corollary}\label{C:conjugators_of_positive_elements}
Let $\alpha\in A_S$ be a non-trivial element that belongs to a proper parabolic subgroup, and is conjugate to a positive element.
Then all positive conjugates of~$\alpha$ belong to a proper standard parabolic subgroup.
Moreover, if~$v$ and~$w$ are positive conjugates of~$\alpha$ with $\mbox{supp}(v)=X$ and $\mbox{supp}(w)=Y$, then for every $x\in A_S$ such that $x^{-1}vx=w$, one has $x^{-1}z_X x = z_Y$ (and hence $x^{-1}A_Xx=A_Y$).
\end{corollary}

\begin{proof}
With the given hypothesis, we know that we can conjugate~$\alpha$ to a positive element $v\in A_X$, where $X\subsetneq S$:
First, compute a conjugate of~$\alpha$ in a proper standard parabolic subgroup, then apply iterated twisted cycling and decycling until a super summit conjugate is obtained; the latter will be positive and, by \autoref{L:cycling_parabolic} and \autoref{L:decycling_parabolic}, contained in a proper standard parabolic subgroup.

By \autoref{P:arrows_in_positive_graph}, there are three types of arrows in~$\mathcal G_{C^+(\alpha)}$ starting at~$v$.
In each of these three cases, consider the target vertex~$v^x$ of an arrow with label~$x$ starting at~$v$:

\begin{enumerate}

\item $x\in A_X$.
In this case $v^x\in A_X$.

\item $x=t\notin X$ where~$t$ is not adjacent to~$X$.
In this case $v^x=v\in A_X$.

\item $x=r_{X,t}$, where $t\notin X$ is adjacent to~$X$.
Then $x^{-1}Xx=Z$ for some proper subset $Z\subsetneq S$.
Hence $v^x\in A_Z$.
\end{enumerate}

Therefore, in every case,~$v^x$ belongs to a proper standard parabolic subgroup.
We can apply the same argument to every vertex of~$\mathcal G_{C^+(\alpha)}$ and, since~$\mathcal G_{C^+(\alpha)}$ is connected, it follows that all vertices in $\mathcal G_{C^+(\alpha)}$ belong to a proper standard parabolic subgroup.

Now denote $X=\mbox{supp}(v)$, let~$x$ be the label of an arrow in~$\mathcal G_{C^+(\alpha)}$ starting at~$v$, and let $Y=\mbox{supp}(v^x)$.

If $x\in A_X$, we have $Y=\mbox{supp}(v^x)\subseteq X$.
If we had $Y\subsetneq X$, then~$v^x$ would be a positive element belonging to a proper standard parabolic subgroup of~$A_X$.
Hence, taking~$A_X$ as the global Artin--Tits group, all positive conjugates of~$v^x$ in~$A_X$ would belong to a proper standard parabolic subgroup of~$A_X$, which is not the case, as~$v$ itself does not satisfy that property.
Hence $Y=X$, so $\mbox{supp}(v^x)=X$.
Moreover, as $x\in A_X$, we have $x^{-1}z_Xx=z_X$.

If $x=t\notin X$ where~$t$ is not adjacent to~$X$, then~$x$ commutes with all letters of~$X$, so we have $v^x=v$, whence $Y=X$, and also $x^{-1}z_Xx=z_X$.

Finally, if $x=r_{X,t}$ where $t\notin X$ is adjacent to~$X$, then $x^{-1}Xx=Z$ for some $Z\subsetneq S$ by \autoref{L_ribbon_permutation}.
As~$v$ contains all letters of~$X$, it follows that~$v^x$ contains all letters of~$Z$, that is, $Z=Y$.
Therefore $x^{-1}Xx=Y$, which implies that $x^{-1}A_X x = A_Y$ and hence $x^{-1} z_X x = z_Y$.

Applying this argument to all arrows in~$\mathcal G_{C^+(\alpha)}$, it follows that the label of any arrow starting at a vertex~$u_0$ and ending at a vertex~$u_1$ conjugates~$z_{\mbox{\scriptsize supp}(u_0)}$ to~$z_{\mbox{\scriptsize supp}(u_1)}$.
This can be extended to paths in~$\mathcal G_{C^+(\alpha)}$:
If a path goes from~$u_0$ to~$u_k$, the element associated to the path conjugates~$z_{\mbox{\scriptsize supp}(u_0)}$ to~$z_{\mbox{\scriptsize supp}(u_k)}$.

Now suppose that~$v$ and~$w$ are two positive conjugates of~$\alpha$, where $\mbox{supp}(v)=X$ and $\mbox{supp}(w)=Y$, and suppose that $x\in A_S$ is such that $x^{-1}vx=w$.
Then~$v$ and~$w$ are vertices of~$\mathcal G_{C^+(\alpha)}$.
Up to multiplying~$x$ by a central power of~$\Delta_S$, we can assume that~$x$ is positive.
Decomposing~$x$ as a product of minimal conjugators, it follows that~$x$ is the element associated to a path in~$\mathcal G_{C^+(\alpha)}$ starting at~$v$ and finishing at~$w$.
Therefore, $x^{-1}z_X x=z_Y$, as we wanted to show.
\end{proof}

\begin{example}\label{E:positive_conjugates_action}
Consider again the situation from \autoref{E:positive_conjugates_arrows} and \autoref{F:positive_conjugates_arrows}.
We see that all positive conjugates of $\sigma_1\sigma_2$ belong to a proper standard parabolic subgroup, namely either to
$\langle\sigma_1,\sigma_2\rangle$, or to $\langle\sigma_2,\sigma_3\rangle$, or to $\langle\sigma_3,\sigma_4\rangle$.
We can also check that the labels of the arrows in the graph $\mathcal G_{C^+(\alpha)}$ conjugate the central elements of the (minimal) standard parabolic subgroups containing the respective conjugates as expected; cf.\ \autoref{F:positive_conjugates_action}.
\end{example}

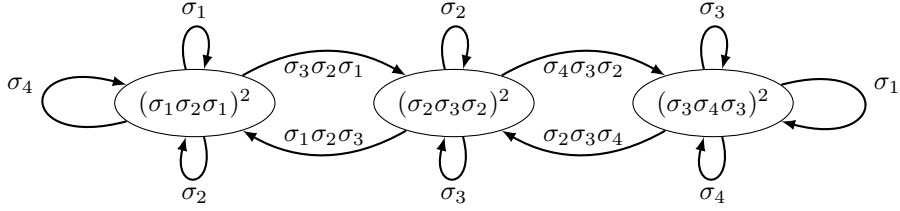
\begin{figure}[h]
\begin{center}
\tikzset{main node/.style={ellipse,draw,minimum size=0.53cm,inner sep=0pt},}
 \begin{tikzpicture}
    \tikzstyle{flecha}=[->, thick,>=latex]
    \node[main node] (1) {$(\sigma_1\sigma_2\sigma_1)^2$\rule[-6pt]{0pt}{18pt}};
    \node[main node] (2) [right = 1.3cm  of 1]  {$(\sigma_2\sigma_3\sigma_2)^2$\rule[-6pt]{0pt}{18pt}};
    \node[main node] (3) [right = 1.3cm  of 2] {$(\sigma_3\sigma_4\sigma_3)^2$\rule[-6pt]{0pt}{18pt}};

    \draw[flecha] (1)to[loop above] node[above]{$\sigma_1$}(1);
    \draw[flecha] (1)to[loop left] node[above left]{$\sigma_4$}(1);
    \draw[flecha] (1)to[bend left] node[below]{$\sigma_3\sigma_2\sigma_1$}(2);
    \draw[flecha] (1)to[loop below] node[below]{$\sigma_2$}(1);

    \draw[flecha] (2)to[bend left] node[above]{$\sigma_1\sigma_2\sigma_3$}(1);
    \draw[flecha] (2)to[loop above] node[above]{$\sigma_2$}(2);
    \draw[flecha] (2)to[loop below] node[below]{$\sigma_3$}(2);
    \draw[flecha] (2)to[bend left] node[below]{$\sigma_4\sigma_3\sigma_2$}(3);

    \draw[flecha] (3)to[loop above] node[above]{$\sigma_3$}(3);
    \draw[flecha] (3)to[bend left] node[above]{$\sigma_2\sigma_3\sigma_4$}(2);
    \draw[flecha] (3)to[loop right] node[above right]{$\sigma_1$}(3);
    \draw[flecha] (3)to[loop below] node[below]{$\sigma_4$}(3);

\end{tikzpicture}
\end{center}
\caption{The action of the conjugating elements from \autoref{F:positive_conjugates_arrows} on the central elements $z_{\,.}$ of the (minimal) standard parabolic subgroups containing the positive conjugates of $\sigma_1\sigma_2$.}\label{F:positive_conjugates_action}
\end{figure}

The result from \autoref{C:conjugators_of_positive_elements} allows to introduce an important concept:

\begin{definition}
Let~$\alpha\in A_S$ be conjugate to a positive element~$\alpha' = \beta^{-1}\alpha \beta \in A_S^+$. Let~$X=\mbox{supp}(\alpha')$. We define the {\em parabolic closure} of $\alpha$ as the subgroup~$P_\alpha = \beta A_X \beta^{-1}$.
\end{definition}

\begin{proposition}\label{P:P_beta_well defined_positive case}
Under the above assumptions, the parabolic closure~$P_\alpha$ is well defined, and it is the smallest parabolic subgroup (by inclusion) containing~$\alpha$.
\end{proposition}

\begin{proof}
Suppose that $\alpha''=\gamma^{-1}\alpha \gamma$ is another positive conjugate of~$\alpha$, and denote $Y=\mbox{supp}(\alpha'')$. We must show that $\beta A_X \beta^{-1} = \gamma A_Y \gamma^{-1}$.
We have $(\alpha')^{\beta^{-1}\gamma}=\alpha''$.
If $\alpha$ belongs to a proper parabolic subgroup, we can apply~\autoref{C:conjugators_of_positive_elements}, so both $A_X$ and $A_Y$ are proper standard parabolic subgroups, the conjugating element $\beta^{-1}\gamma$ maps $z_X$ to~$z_Y$, and then it maps $A_X$ to $A_Y$ by~\autoref{L:parabolic=central_element}.
Therefore $\beta A_X \beta^{-1} = \gamma A_Y \gamma^{-1}$, that is, $P_\alpha$ is well defined.
If $\alpha$ does not belong to a proper parabolic subgroup, none of its conjugates does, hence $A_X=A_Y=A_S$. We then have $\beta A_X \beta^{-1} = \beta A_S \beta^{-1} = A_S = \gamma A_S \gamma^{-1} = \gamma A_Y \gamma^{-1}$, so $P_\alpha$ is also well defined (and equal to~$A_S$) in this case.

Now let $P$ be a parabolic subgroup containing~$\alpha$. Let $x\in A_S$ be such that $x^{-1} P x$ is standard, that is, $x^{-1}Px=A_Z$ for some $Z\subseteq S$. Then $x^{-1}\alpha x\in A_Z$ and we can obtain another conjugate $\alpha'=y^{-1}x^{-1}\alpha xy \in A_Z \cap SSS(\alpha)$, where $y\in A_Z$, by iterated twisted cycling and iterated decycling. Since $\alpha$ is conjugate to a positive element, all elements in $SSS(\alpha)$ are positive, so $\alpha'$ is positive. Hence, if we denote $X=\mbox{supp}(\alpha')$, we have $A_X\subseteq A_Z$. Conjugating back, we have $P_\alpha = xy A_X y^{-1}x^{-1} \subseteq xy A_Z y^{-1}x^{-1} = x A_Z x^{-1} = P$. Hence $P_\alpha$ is contained in any parabolic subgroup containing~$\alpha$, as we wanted to show.
\end{proof}

\section{Parabolic closure of an arbitrary element}

In the previous section we defined the parabolic closure of a given element $\alpha\in A_S$, provided that $\alpha$ is conjugate to a positive element. In this section we will see how to extend this definition to every element $\alpha\in A_S$. That is, we will see that for every element $\alpha\in A_S$ there is a parabolic subgroup $P_\alpha$ which is the smallest parabolic subgroup (by inclusion) containing~$\alpha$.

Instead of using the positive conjugates of $\alpha$ (which may not exist), we will define~$P_\alpha$ by using one of the summit sets we defined earlier: $RSSS_\infty(\alpha)$.  We will see that the support of the elements in $RSSS_\infty(\alpha)$ is also {\it preserved by conjugation}, so it can be used to define the smallest parabolic subgroup containing~$\alpha$.

\begin{definition}
Let $\alpha\in A_S$.
Let $\alpha'\in RSSS_\infty(\alpha)$, where $\alpha'=\beta^{-1}\alpha \beta$.
If we denote $Z=\mbox{supp}(\alpha')$, we define the parabolic closure of $\alpha$ as $P_\alpha = \beta A_Z \beta^{-1}$.
\end{definition}

The next result shows~\autoref{T:minimal_parabolic_containing_an_element}.

\begin{proposition}\label{P:minimalContainingParabolic}
Under the above assumptions, the parabolic closure~$P_\alpha$ is well defined, and it is the smallest parabolic subgroup (by inclusion) containing~$\alpha$.
\end{proposition}

\begin{proof}
Write $\alpha'=x^{-1}y$ in $np$-normal form, and recall that $Z=\mbox{supp}(\alpha')=\mbox{supp}(x)\cup \mbox{supp}(y)$.
If $\alpha$ does not belong to a proper parabolic subgroup, none of its conjugates does, hence $\mbox{supp}(v)=S$ for every $v\in RSSS_\infty(\alpha)$, so $P_\alpha=A_S$ is well defined, and it is indeed the smallest parabolic subgroup containing~$\alpha$. Hence we can assume that $\alpha$ belongs to some proper parabolic subgroup.

If $\alpha$ is conjugate to a positive element, then all elements in $SSS(\alpha)$ will be positive. As $RSSS_\infty(\alpha)\subseteq RSSS(\alpha)\subseteq SSS(\alpha)$, it follows that $\alpha'$ is positive and $Z=\mbox{supp}(\alpha')$.
Therefore the above definition of $P_\alpha$ coincides with the definition we gave in the previous section.
Hence, by~\autoref{P:P_beta_well defined_positive case}, $P_\alpha$ is well defined and it is the smallest parabolic subgroup containing~$\alpha$.

Suppose that $\alpha^{-1}$ is conjugate to a positive element.
As one has $\beta\in SSS(\alpha)$ if and only if $\beta^{-1}\in SSS(\alpha^{-1})$, the inverse of every element of $RSSS(\alpha)\subseteq SSS(\alpha)$ is positive. In this case $Z=\mbox{supp}\left((\alpha')^{-1}\right)$, and $P_\alpha$ coincides with the definition of $P_{\alpha^{-1}}$ in the previous section. Hence $P_\alpha$ is well defined, and it is the smallest parabolic subgroup containing $\alpha^{-1}$, thus the smallest one containing~$\alpha$.

We can then assume that $\alpha$ belongs to some proper parabolic subgroup, and that the $np$-normal form of $\alpha'$ has the form $\alpha'=x^{-1}y= x_s^{-1}\cdots x_1^{-1} y_1\cdots y_t$, with $s,t>0$ (that is, $x,y\neq 1$).

Let $N=\max(s,t)$, and let us use the Garside structure $(A_S, A_S^+, \Delta_S^N)$. With respect to this structure, $x$ and~$y$ are simple elements and, as $\alpha'\in RSSS_\infty(\alpha)$, both $\alpha'$ and $(\alpha')^{-1}$ belong to their respective ultra summit sets.

In order to show that $P_\alpha$ is well defined, let $\alpha''= \gamma^{-1}\alpha \gamma\in RSSS_\infty(\alpha)$.  We have to show that every element~$g$ conjugating $\alpha'$ to~$\alpha''$, conjugates $A_Z$ to $A_U$, where $Z=\mbox{supp}(\alpha')$ and $U=\mbox{supp}(\alpha'')$. We will show this by constructing positive elements with supports $Z$ and~$U$, respectively, which are also conjugate by~$g$; then the claim follows by \autoref{C:conjugators_of_positive_elements}.

As we are using the Garside structure $(A_S, A_S^+, \Delta_S^N)$, we have that $\alpha'=x^{-1}y$ in $np$-normal form, where $x$ and $y$ are simple elements, and also $\alpha''=u^{-1}v$ in $np$-normal form, where $u$ and $v$ are simple elements, since both $\alpha'$ and $\alpha''$ belong to $SSS_N(\alpha)$, so they have the same infimum and supremum.

Let then $g$ be an element such that $g^{-1} \alpha' g = \alpha''$, and recall that $\alpha',\alpha''\in RSSS_\infty(\alpha) \subseteq USS_\infty(\alpha) \subseteq USS_N(\alpha)$. By~\autoref{C:conjugated_conjugating_elements_twisted}, there is a positive integer $M$ such that $g^{-1}\widetilde C_M(\alpha') g = \widetilde C_M(\alpha'')$, where $\widetilde C_M(\alpha')$ (resp.\ $\widetilde C_M(\alpha'')$) is the product of the conjugating elements for $M$ consecutive twisted cyclings of~$\alpha'$ (resp.\ $\alpha''$) with respect to the Garside structure $(A_S, A_S^+, \Delta_S^N)$.

By~\autoref{L:twisted_cycling}, $\widetilde C_M(\alpha')$ is the inverse of a positive element, say $w_1(\alpha'):=\widetilde C_M(\alpha')^{-1}$. Recall that $\alpha'\in A_Z$, hence the factors composing $\widetilde C_M(\alpha')$ belong to~$A_Z$, and then $w_1(\alpha')\in A_Z^+$. In the same way, $\widetilde C_M(\alpha'')$ is the inverse of a positive element, say $w_1(\alpha''):=\widetilde C_M(\alpha'')^{-1}\in A_U^+$. We then have $g^{-1} w_1(\alpha') g = w_1(\alpha'')$, where $w_1(\alpha')\in A_Z^+$ and $w_1(\alpha'')\in A_U^+$.

Now, from $g^{-1} \alpha' g = \alpha''$ we obtain $g^{-1} (\alpha')^{-1} g = (\alpha'')^{-1}$. Since $\alpha',\alpha''\in RSSS_\infty(\alpha)$, it follows that $(\alpha')^{-1},(\alpha'')^{-1}\in USS_\infty(\alpha^{-1}) \subseteq USS_N(\alpha^{-1})$.
Thus, we can apply~\autoref{C:conjugated_conjugating_elements_twisted} and obtain that there is a positive integer~$T$ such that $g^{-1}\widetilde C_T((\alpha')^{-1}) g = \widetilde C_T((\alpha'')^{-1})$ in the same way as above. If we denote $w_2(\alpha') = \widetilde C_T((\alpha')^{-1})^{-1}$ and $w_2(\alpha'') = \widetilde C_T((\alpha'')^{-1})^{-1}$, we have $g^{-1} w_2(\alpha') g = w_2(\alpha'')$, where $w_2(\alpha')\in A_Z^+$ and $w_2(\alpha'')\in A_U^+$.

Let us denote $w(\alpha')=w_1(\alpha')w_2(\alpha')\in A_Z^+$ and $w(\alpha'')=w_1(\alpha'')w_2(\alpha'')\in A_U^+$.
By construction, $g^{-1} w(\alpha') g = w(\alpha'')$.
We will now show that $\mbox{supp}(w(\alpha'))=Z$ and $\mbox{supp}(w(\alpha''))=U$.

Notice that the conjugating element for twisted cycling of $\alpha'=x^{-1}y$ using the Garside structure $(A_S, A_S^+, \Delta_S^N)$ is~$x^{-1}$. By \autoref{L:twisted_cycling}, $x$ is a suffix of~$w_1(\alpha')$. On the other hand, the conjugating element for twisted cycling of $(\alpha')^{-1}=y^{-1}x$ is~$y^{-1}$. Hence, $y$ is a suffix of $w_2(\alpha')$. This implies that $Z=\mbox{supp}(\alpha')=\mbox{supp}(x)\cup \mbox{supp}(y)\subseteq \mbox{supp}(w_1(\alpha'))\cup \mbox{supp}(w_2(\alpha'))=\mbox{supp}(w(\alpha'))\subseteq Z$, whence $\mbox{supp}(w(\alpha'))=Z$. In the same way it follows that $\mbox{supp}(w(\alpha''))=U$.

We can then apply~\autoref{C:conjugators_of_positive_elements}, since $g^{-1} w(\alpha') g = w(\alpha'')$, to conclude that $g^{-1} A_Z g =A_U$, as we wanted to show. This means that $P_\alpha$ is well defined, as taking $g= \beta^{-1}\gamma$ (a conjugating element from $\alpha'$ to~$\alpha''$) one has:
$$
 \beta A_Z \beta^{-1} = \beta g A_U g^{-1} \beta^{-1} = \gamma A_U \gamma^{-1},
$$
hence one can equally use either $\alpha'$ or $\alpha''$ to define~$P_\alpha$.

Now let us prove that~$P_\alpha$ is the smallest parabolic subgroup (for inclusion) containing~$\alpha$. Suppose that $P$ is a parabolic subgroup containing~$\alpha$, and let $a\in A_S$ be an element such that $a^{-1}Pa=A_X$ is standard. Then $a^{-1}\alpha a\in A_X$. We can now apply iterated twisted cyclings and iterated decyclings to this element (in all needed Garside structures), so that the resulting element, say $\widehat \alpha$, belongs to $RSSS_\infty(\alpha)$. The product of all conjugating elements, call it~$b$, will belong to~$A_X$, so we will have $\widehat\alpha = b^{-1}a^{-1} \alpha ab \in A_X \cap RSSS_\infty(\alpha)$.

Let $Y=\mbox{supp}(\widehat \alpha)$. By definition, we have $P_\alpha = ab A_Y b^{-1}a^{-1}$. But on the other hand, as $\widehat \alpha\in A_X$, all letters in the $np$-normal form of $\widehat\alpha$ belong to~$A_X$. Hence $A_Y\subseteq A_X$, and we finally have:
$$
   P_\alpha = ab A_Y b^{-1}a^{-1} \subseteq ab A_X b^{-1}a^{-1} = a A_X a^{-1} =P.
$$
Therefore, $P_\alpha$ is contained in every parabolic subgroup containing $\alpha$, as we wanted to show.
\end{proof}

\section{Parabolic subgroups, powers and roots}

In this section we will see that the parabolic closure $P_\alpha$ of $\alpha\in A_S$ behaves as expected under conjugation, taking powers and taking roots.
The behavior under conjugation follows directly from the definition:

\begin{lemma}\label{L:associated_parabolic_and_conjugacy}
For every $\alpha,x \in A_S$, one has $P_{x^{-1}\alpha x} = x^{-1}P_\alpha x$.
\end{lemma}

\begin{proof}
Let $\alpha'= \beta^{-1}\alpha \beta \in RSSS_\infty(\alpha)$. Then $\alpha'=\beta^{-1}x(x^{-1}\alpha x)x^{-1}\beta$.

If $X=\mbox{supp}(\alpha')$, by definition we have $P_\alpha=\beta A_X \beta^{-1}$, and also
$$
 P_{x^{-1}\alpha x}=x^{-1}\beta A_X \beta^{-1} x = x^{-1} P_\alpha x.
$$
\end{proof}

The behavior of $P_\alpha$ when taking powers or roots is not so easy, but it is also as expected:

\begin{theorem}\label{T:parabolic_for_powers_and_roots}
Let $A_S$ be an Artin--Tits group of spherical type. If $\alpha\in A_S$ and~$m$ is a nonzero integer, then $P_{\alpha^m}=P_\alpha$.
\end{theorem}

\begin{proof}
    By~\autoref{L:associated_parabolic_and_conjugacy}, we can conjugate $\alpha$ to assume that $\alpha\in RSSS_\infty(\alpha)$. We can further conjugate $\alpha$ by the conjugating elements for iterated twisted cycling and iterated decycling of its $m$-th power (for all needed Garside structures), in order to take $\alpha^m$ to $RSSS_\infty(\alpha^m)$. By \autoref{L:conjugating_elements_to_summit_sets}, these conjugating elements are the meet of two elements which conjugate $\alpha$ to other elements in  $RSSS_\infty(\alpha)$. Hence, by the convexity of $RSSS_\infty(\alpha)$,
 they maintain $\alpha$ in $RSSS_\infty(\alpha)$. In summary, up to conjugacy we can assume that $\alpha\in RSSS_\infty(\alpha)$ and $\alpha^m\in RSSS_\infty(\alpha^m)$. We can also assume that $m$ is positive, as $P_{\alpha^{-m}}=P_{\alpha^m}$ by definition.

Under these assumptions, the parabolic closures of $\alpha$ and $\alpha^m$ will be determined by their corresponding supports. Hence, we must show that $\mbox{supp}(\alpha)=\mbox{supp}(\alpha^m)$.

If either $\alpha$ or $\alpha^{-1}$ is positive, the result is clear. We can then assume that this is not the case, that is, we have $\alpha=x_1^{-1}y_1$ in $np$-normal form with $x_1,y_1\neq 1$. Moreover, using a suitable Garside structure, we can assume that $x_1$ and $y_1$ are simple elements.

We will now use the following result~\cite[Theorem 2.9]{BGG}: If $\alpha\in USS(\alpha)$, $\inf(\alpha)=p$, $\ell(\alpha)>1$ and $m\geq 1$, one has
\begin{equation}\label{EQ:IteratedCycling}
    \alpha^m \Delta^{-mp} \wedge \Delta^m = C_m(\alpha).
\end{equation}

In other words, if we consider the positive element $\alpha^m \Delta^{-mp}$, and compute the first $m$ factors in its left normal form (including any $\Delta$ factors), the product of these $m$ factors equals the product of the first $m$ conjugating elements for iterated cyclings of~$\alpha$.

We point out that we are using a Garside structure $(A_S,A_S^+,\Delta_S^N)$ such that $\alpha=x_1^{-1}y_1$, where $x_1$ and $y_1$ are non-trivial simple elements.
So the $\Delta$ in \eqref{EQ:IteratedCycling} means $\Delta_S^N$ in our case.
Also, we remark that $\alpha$ satisfies the hypotheses for~\eqref{EQ:IteratedCycling}, as $\alpha \in RSSS_\infty(\alpha)\subseteq USS_\infty(\alpha)\subseteq USS_N(\alpha)$, and the left normal form of $\alpha$ is $\Delta^{-1}\widetilde x_1 y_1$, so $\inf(\alpha)=-1$ and $\ell(\alpha)=2$.

We will restate the above result from~\cite{BGG} in terms of $np$-normal forms. It turns out that the statement will become much nicer.

\begin{figure}
\[
\xymatrix@C=12mm@R=12mm{
 & & & &
\\
  \ar[u]_{x_1} \ar[r]^{y_1} & & & &
\\
  \ar[u]_{x_2} \ar[r]^{y_2}
& \ar[u]_{x_1} \ar[r]^{y_1} & & &
\\
  \ar[u]_{x_3} \ar[r]^{y_3}
& \ar[u]_{x_2} \ar[r]^{y_2}
& \ar[u]_{x_1} \ar[r]^{y_1} & &
\\
  \ar[u]_{x_4} \ar[r]^{y_4}
& \ar[u]_{x_3} \ar[r]^{y_3}
& \ar[u]_{x_2} \ar[r]^{y_2}
& \ar[u]_{x_1} \ar[r]^{y_1} &
}
\]
\caption{How to transform $\alpha^4=(x_1^{-1}y_1)^4$ into $x_1^{-1}x_2^{-1}x_3^{-1} x_4^{-1}y_4 y_3 y_2 y_1$. Around each square, the product $x_i^{-1}y_i$ is the $np$-normal form of $y_{i-1}x_{i-1}^{-1}$.}
\label{F:stairs}
\end{figure}
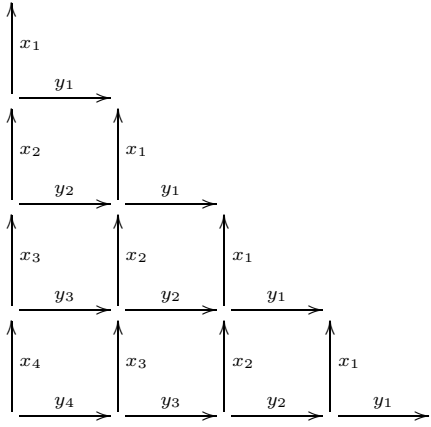

Write $\alpha^m = x_1^{-1} y_1 x_1^{-1} y_1 \cdots x_1^{-1} y_1$ (see~\autoref{F:stairs}). Notice that applying a twisted cycling to $\alpha$ means to conjugate it by $x_1^{-1}$, and one obtains $\alpha_2=y_1 x_1^{-1}$, whose $np$-normal form will be of the form $\alpha_2=x_2^{-1}y_2$. We then have:
$$
   \alpha^m =  x_1^{-1} (x_2^{-1} y_2 x_2^{-1} y_2 \cdots x_2^{-1} y_2) y_1
$$
Now we can apply twisted cycling to $\alpha_2$ (conjugating it by $x_2^{-1}$), and we obtain $\alpha_3 = x_3^{-1}y_3$. Then we see that:
$$
  \alpha^m = x_1^{-1}x_2^{-1}(x_3^{-1}y_3\cdots x_3^{-1}y_3)y_2y_1.
$$
Repeating the process $m$ times, we finally obtain:
$$
  \alpha^m = x_1^{-1}x_2^{-1}\cdots x_m^{-1}y_m\cdots y_2y_1.
$$
That is, we have written $\alpha^m$ as a product of a negative times a positive element, where the negative one is the product of the first $m$ conjugating elements for iterated twisted cycling of~$\alpha$. We will see that~\cite[Theorem 2.9]{BGG} is equivalent to the fact that no cancellation occurs between $x_1^{-1}x_2^{-1}\cdots x_m^{-1}$ and $y_m\cdots y_2y_1$.

Indeed, on the one hand, as $p=-1$, we have the element $\alpha^m\Delta^{-pm} = \alpha^m \Delta^{m}$, which is:
$$
  \alpha^m \Delta^{m} = x_1^{-1}\cdots x_m^{-1} y_m\cdots y_1 \Delta^{m}
$$
$$
  = (x_1^{-1}\Delta)(\Delta^{-1}x_2^{-1}\Delta^2) \cdots (\Delta^{m-1}x_m^{-1}\Delta^m)  \tau^m(y_m\cdots y_1).
$$
On the other hand, for $i=1,\ldots,m$, we have $\alpha_{i}=\widetilde c^{i-1}(\alpha)= \tau^{-i+1}\circ c^{i-1}(\alpha)$, so $c^{i-1}(\alpha)= \tau^{i-1}(\alpha_{i})=\tau^{i-1}(x_i)^{-1}\tau^{i-1}(y_i)$. This means that the $i$-th conjugating element for iterated cycling of $\alpha$ is $\tau^{i-1}(x_i)^{-1} \Delta = \Delta^{i-1} x_i^{-1} \Delta^i$.

Hence, by~\cite[Theorem 2.9]{BGG} the product of the first $m$ factors of the left normal form of $\alpha^m \Delta^{m}$ is equal to $(x_1^{-1}\Delta)(\Delta^{-1}x_2^{-1}\Delta^2) \cdots (\Delta^{m-1}x_m^{-1}\Delta^m)$, that is, to $x_1^{-1}x_2^{-1}\cdots x_m^{-1}\Delta^m$. In other words:
$$
  \alpha^m \Delta^{m} \wedge \Delta^m = x_1^{-1}x_2^{-1}\cdots x_m^{-1}\Delta^m.
$$
As the greatest common prefix is preserved by multiplication on the right by any power of~$\Delta$, we obtain:
$$
    \alpha^m \wedge 1 = x_1^{-1}x_2^{-1}\cdots x_m^{-1}.
$$
But the biggest common prefix of an element and 1 is precisely the negative part of its $np$-normal form, so this shows that there is no cancellation between $x_1^{-1}x_2^{-1}\cdots x_m^{-1}$ and $y_m\cdots y_2y_1$, as we claimed.

But then:
\begin{align*}
\mbox{supp}(\alpha^m) &= \mbox{supp}(x_m\cdots x_1)\cup \;\mbox{supp}(y_m\cdots y_1) \\
                      &\supseteq \mbox{supp}(x_1)\cup \;\mbox{supp}(y_1)=\mbox{supp}(\alpha).
\end{align*}

Since it is clear that $\mbox{supp}(\alpha^m)\subseteq \mbox{supp}(\alpha)$ (as no new letters can appear when computing the $np$-normal form of $\alpha^m$ starting with $m$ copies of the $np$-normal form of $\alpha$), we finally have $\mbox{supp}(\alpha^m)= \mbox{supp}(\alpha)$, as we wanted to show.
\end{proof}

This behavior of $P_\alpha$ with respect to taking powers or roots allows us to show an interesting consequence: All roots of an element in a parabolic subgroup belong to the same parabolic subgroup.

\begin{corollary}\label{C:roots_in_parabolic}
Let $A_S$ be an Artin--Tits group of spherical type. If $\alpha$ belongs to a parabolic subgroup~$P$, and $\beta\in A_S$ is such that $\beta^m=\alpha$ for some nonzero integer~$m$, then $\beta\in P$.
\end{corollary}

\begin{proof}
Since $\alpha$ is a power of $\beta$,~\autoref{T:parabolic_for_powers_and_roots} tells us that $P_\alpha=P_\beta$. Now $\alpha \in P$ implies $P_\alpha\subseteq P$, as $P_\alpha$ is the minimal parabolic subgroup containing~$\alpha$. Hence $\beta \in P_\beta =P_\alpha \subseteq P$.
\end{proof}

\section{Intersection of parabolic subgroups}

In this section we will show one of the main results of this paper: The intersection of two parabolic subgroups in an Artin--Tits group of spherical type $A_S$ is also a parabolic subgroup.

We will use the parabolic closure $P_\alpha$ of an element $\alpha\in A_S$, but we will also need some technical results explaining how the left normal form of a positive element, in which some factors equal $\Delta_X$ for some $X\subseteq S$, behaves when it is multiplied by another element:

\begin{lemma}
Let $X\subseteq S$ be nonempty, and let $\alpha\in A_S^+$ be a simple element such that $\Delta_X \alpha$ is simple. Then there is $Y\subseteq S$ and a decomposition $\alpha=\rho \beta$ in $A_S^+$, such that $X\rho = \rho Y$ and the left normal form of $\Delta_X (\Delta_X \alpha)$ is $(\rho \Delta_Y)(\Delta_Y \beta)$.
\end{lemma}

\begin{proof}
We proceed by induction on the length $|\alpha|$ of $\alpha$ as a word over~$S$. If $|\alpha|=0$ then $\Delta_X (\Delta_X \alpha)= \Delta_X\Delta_X$, which is already in left normal form, and the result holds taking $Y=X$ and $\rho=\beta=1$. Suppose then that $|\alpha|>0$ and the result is true when $\alpha$ is shorter.

Let $Z\subseteq S$ be the set of initial letters of $\Delta_X \alpha$. That is, $Z=\{\sigma_i\in S : \sigma_i\preccurlyeq \Delta_X \alpha\}$. It is clear that $X\subseteq Z$, as every letter of $X$ is a prefix of~$\Delta_X$.

Notice that the set of final letters of $\Delta_X$ is precisely~$X$. Hence, if $Z=X$ the decomposition $\Delta_X(\Delta_X \alpha)$ is already in left normal form, so the result holds taking $X=Y$, $\rho=1$ and $\beta= \alpha$.

Suppose on the contrary that there exists some $t\in Z$ which is not in~$X$. Then $t\preccurlyeq \Delta_X \alpha$ so $\Delta_X r_{X,t} = t\vee \Delta _X\preccurlyeq \Delta_X\alpha$. Cancelling $\Delta_X$ from the left we have $r_{X,t}\preccurlyeq \alpha$, so $\alpha=r_{X,t}\alpha_1$.

We know that $X r_{X,t}= r_{X,t} T$ for some $T\subsetneq S$. Hence:
$$
\Delta_X (\Delta_X \alpha) = \Delta_X (\Delta_X r_{X,t}\alpha_1) = (r_{X,t}\Delta_T)(\Delta_T\alpha_1).
$$
Now consider the element $\Delta_T(\Delta_T\alpha_1)$. As $\alpha_1$ is shorter than $\alpha$, we can apply the induction hypothesis to obtain that $\alpha_1 = \rho_1 \beta$, where $T\rho_1 = \rho_1 Y$ for some $Y\subsetneq S$, and the left normal form of $\Delta_T(\Delta_T\alpha_1)$ is $(\rho_1 \Delta_Y)(\Delta_Y \beta)$.

We claim that the left normal form of $\Delta_X (\Delta_X \alpha)$ is $(\rho \Delta_Y)(\Delta_Y \beta)$, where $\rho =r_{X,t}\rho_1$. First, we have:
$$
  \Delta_X (\Delta_X \alpha) = \Delta_X (\Delta_X r_{X,t}\rho_1 \beta) = (r_{X,t} \Delta_T)(\Delta_T \rho_1 \beta)
$$
$$
  = (r_{X,t}\rho_1 \Delta_Y)(\Delta_Y \beta)= (\rho \Delta_Y)(\Delta_Y \beta).
$$
Note that $\rho \Delta_Y = \Delta_X \rho$ is simple, as it is a prefix of $\Delta_X \alpha$ which is simple. Finally, the set of final letters of $\rho \Delta_Y$ contains the set of final letters of its suffix $\rho_1 \Delta_Y$, which in turn contains the set of initial letters of $\Delta_Y \beta$ (as the product $(\rho_1 \Delta_Y)(\Delta_Y \beta)$ is in left normal form). So the claim holds.

Now it just remains to notice that $X \rho = X r_{X,t} \rho_1 = r_{X,t} T \rho_1 = r_{X,t} \rho_1 Y = \rho Y$, to finish the proof.
\end{proof}

\begin{lemma}~\label{L:Deltas_times_something}
Let $X\subseteq S$ be nonempty, and let $m>r>0$. Let $\alpha\in A_S^+$ be such that $\sup(\alpha)=r$, and let $x_1 x_2 \cdots x_{m+r}$ be the left normal form of $(\Delta_X)^m \alpha$ (where some of the first factors can be equal to~$\Delta_S$ and some of the last factors can be trivial). Then there is $Y\subseteq S$ and a decomposition $\alpha=\rho \beta$ in~$A_S^+$, such that $X \rho = \rho Y$ and:
\begin{enumerate}

  \item $x_1\cdots x_r= (\Delta_X)^r \rho = \rho \: (\Delta_Y)^r$.

  \item  $x_i=\Delta_Y$ for $i=r+1,\ldots,m-1$.

  \item $x_m\cdots x_{m+r}=\Delta_Y \beta$.

\end{enumerate}
\end{lemma}

\begin{proof}
Suppose first that $\alpha$ is a simple element, so $r=1$.
By the domino rule \cite[Definition III\,1.57, Proposition V\,1.52]{DDGKM} (see also \cite[Lemma 1.32]{DG}), the left normal form of $(\Delta_X)^m \alpha$ is computed as follows:
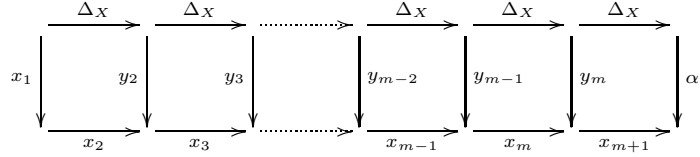
\[
\xymatrix@C=12mm@R=12mm{
  \ar[r]^{\Delta_X} \ar[d]_{x_1}
& \ar[r]^{\Delta_X} \ar[d]_{y_2}
& \ar@{.>}[r]  \ar[d]_{y_3}
& \ar[r]^{\Delta_X} \ar[d]^{y_{m-2}}
& \ar[r]^{\Delta_X} \ar[d]^{y_{m-1}}
& \ar[r]^{\Delta_X} \ar[d]^{y_{m}}
& \ar[d]^{\alpha}
\\
  \ar[r]_{x_2}
& \ar[r]_{x_3}
& \ar@{.>}[r]
& \ar[r]_{x_{m-1}}
& \ar[r]_{x_{m}}
& \ar[r]_{x_{m+1}}
&
}
\]
where the $x_i$ and the $y_i$ are defined from right to left, in the following way: First, the left normal form of $\Delta_X \alpha$ is $y_{m}x_{m+1}$ (here $x_{m+1}$ could be trivial). Then the left normal form of $\Delta_X y_{m}$ is $y_{m-1}x_{m}$, and so on. Around each square, the down-right path represents the left normal form of the right-down path.

By construction, $\Delta_X \preccurlyeq y_m$. Hence, by the previous lemma, the left normal form of $\Delta_X y_m$ is ($\rho \Delta_Y)(\Delta_Y \beta_1)$, for some $Y\subseteq S$ and some~$\rho$ which conjugates~$X$ to~$Y$. Hence $y_{m-1}=\rho \Delta_Y = \Delta_X \rho$, and $x_{m}= \Delta_Y \beta_1$.

Now, if some $y_k=\Delta_X\rho$, as it is clear that the left normal form of $\Delta_X (\Delta_X \rho)$ is  $(\rho \Delta_Y)(\Delta_Y)$, it follows that $y_{k-1}=\rho\Delta_Y=\Delta_X\rho$ and $x_k=\Delta_Y$. Therefore, the above diagram is actually as follows:
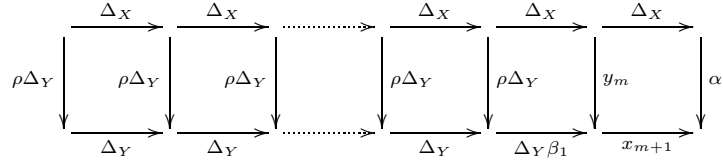
\[
\xymatrix@C=12mm@R=12mm{
  \ar[r]^{\Delta_X} \ar[d]_{\rho \Delta_Y}
& \ar[r]^{\Delta_X} \ar[d]_{\rho\Delta_Y}
& \ar@{.>}[r]  \ar[d]_{\rho\Delta_Y}
& \ar[r]^{\Delta_X} \ar[d]^{\rho\Delta_Y}
& \ar[r]^{\Delta_X} \ar[d]^{\rho\Delta_Y}
& \ar[r]^{\Delta_X} \ar[d]^{y_m}
& \ar[d]^{\alpha}
\\
  \ar[r]_{\Delta_Y}
& \ar[r]_{\Delta_Y}
& \ar@{.>}[r]
& \ar[r]_{\Delta_Y}
& \ar[r]_{\Delta_Y \beta_1}
& \ar[r]_{x_{m+1}}
&
}
\]
As $\rho\Delta_Y\Delta_Y = \Delta_X\Delta_X\rho$, we have $\alpha= \rho \beta_1 x_{m+1}$. This shows the result for $r=1$.

The case $r>1$ follows from the above one and the domino rule. If $\alpha=\alpha_1\cdots \alpha_r$ in left normal form, the left normal form of $(\Delta_X)^m \alpha$ is computed by completing the squares in the diagram in \autoref{F:domino} (row by row, from right to left), where the down-right path is the left normal form of the right-down path, around each square.

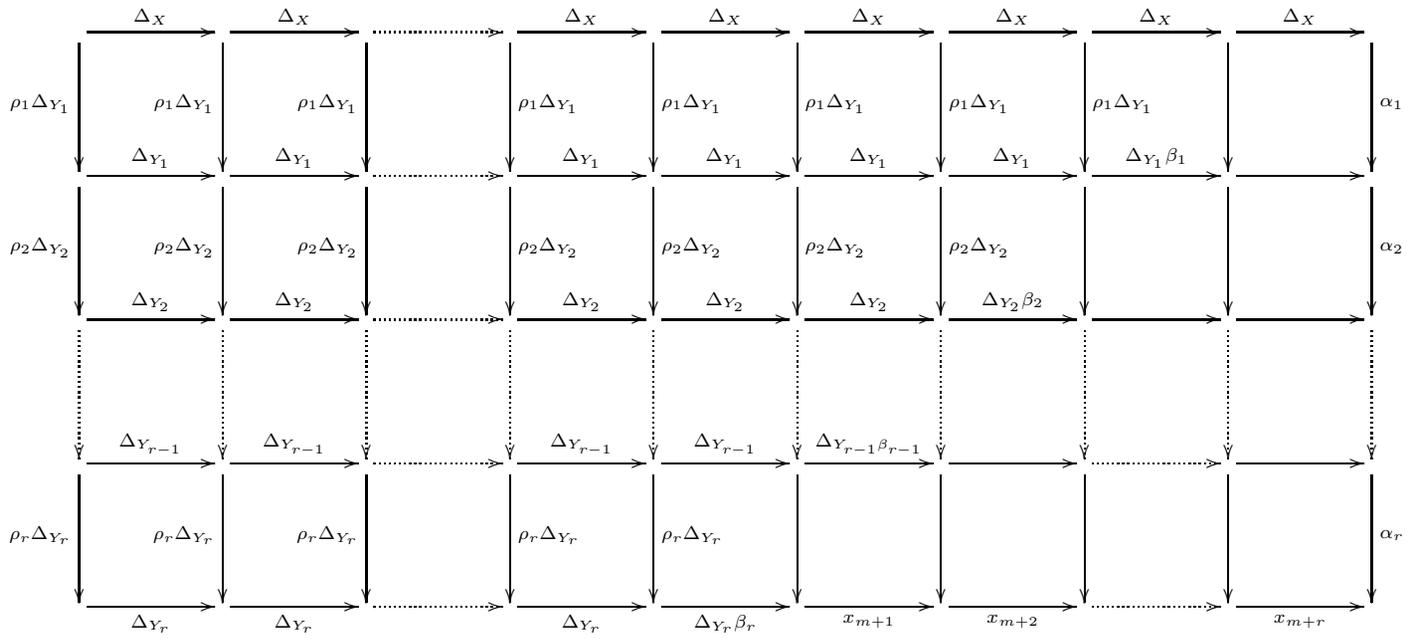
\begin{sidewaysfigure}
\[
\xymatrix@C=17mm@R=17mm{
  \ar[r]^{\Delta_X} \ar[d]_{\rho_1 \Delta_{Y_1}}
& \ar[r]^{\Delta_X} \ar[d]_{\rho_1\Delta_{Y_1}}
& \ar@{.>}[r]  \ar[d]_{\rho_1\Delta_{Y_1}}
& \ar[r]^{\Delta_X} \ar[d]^{\rho_1\Delta_{Y_1}}
& \ar[r]^{\Delta_X} \ar[d]^{\rho_1\Delta_{Y_1}}
& \ar[r]^{\Delta_X} \ar[d]^{\rho_1\Delta_{Y_1}}
& \ar[r]^{\Delta_X} \ar[d]^{\rho_1\Delta_{Y_1}}
& \ar[r]^{\Delta_X} \ar[d]^{\rho_1\Delta_{Y_1}}
& \ar[r]^{\Delta_X} \ar[d]
& \ar[d]^{\alpha_1}
\\
  \ar[r]^{\Delta_{Y_1}} \ar[d]_{\rho_2 \Delta_{Y_2}}
& \ar[r]^{\Delta_{Y_1}} \ar[d]_{\rho_2 \Delta_{Y_2}}
& \ar@{.>}[r] \ar[d]_{\rho_2\Delta_{Y_2}}
& \ar[r]^{\Delta_{Y_1}} \ar[d]^{\rho_2 \Delta_{Y_2}}
& \ar[r]^{\Delta_{Y_1}} \ar[d]^{\rho_2 \Delta_{Y_2}}
& \ar[r]^{\Delta_{Y_1}} \ar[d]^{\rho_2 \Delta_{Y_2}}
& \ar[r]^{\Delta_{Y_1}} \ar[d]^{\rho_2 \Delta_{Y_2}}
& \ar[r]^{\Delta_{Y_1} \beta_1} \ar[d]
& \ar[r] \ar[d]
& \ar[d]^{\alpha_2}
\\
  \ar[r]^{\Delta_{Y_2}} \ar@{.>}[d]
& \ar[r]^{\Delta_{Y_2}} \ar@{.>}[d]
& \ar@{.>}[r] \ar@{.>}[d]
& \ar[r]^{\Delta_{Y_2} } \ar@{.>}[d]
& \ar[r]^{\Delta_{Y_2} } \ar@{.>}[d]
& \ar[r]^{\Delta_{Y_2} } \ar@{.>}[d]
& \ar[r]^{\Delta_{Y_2} \beta_2} \ar@{.>}[d]
& \ar[r] \ar@{.>}[d]
& \ar[r] \ar@{.>}[d]
& \ar@{.>}[d]
\\
  \ar[r]^{\Delta_{Y_{r-1}}} \ar[d]_{\rho_r \Delta_{Y_r}}
& \ar[r]^{\Delta_{Y_{r-1}}} \ar[d]_{\rho_r \Delta_{Y_r}}
& \ar@{.>}[r] \ar[d]_{\rho_r\Delta_{Y_r}}
& \ar[r]^{\Delta_{Y_{r-1}}} \ar[d]^{\rho_r \Delta_{Y_r}}
& \ar[r]^{\Delta_{Y_{r-1}}} \ar[d]^{\rho_r \Delta_{Y_r}}
& \ar[r]^{\Delta_{Y_{r-1}\beta_{r-1}}} \ar[d]
& \ar[r] \ar[d]
& \ar@{.>}[r] \ar[d]
& \ar[r] \ar[d]
& \ar[d]^{\alpha_r}
\\
  \ar[r]_{\Delta_{Y_r}}
& \ar[r]_{\Delta_{Y_r}}
& \ar@{.>}[r]
& \ar[r]_{\Delta_{Y_r}}
& \ar[r]_{\Delta_{Y_r}\beta_r}
& \ar[r]_{x_{m+1}}
& \ar[r]_{x_{m+2}}
& \ar@{.>}[r]
& \ar[r]_{x_{m+r}}
&
}
\]
\caption{Computing the left normal form of $(\Delta_X)^m \alpha$ in the proof of \autoref{L:Deltas_times_something}.}
\label{F:domino}
\end{sidewaysfigure}

By construction, the subsets $X=Y_0,Y_1,\cdots, Y_r=Y$ of~$S$ and the elements $\rho_1,\ldots,\rho_r$ satisfy $Y_{i-1}\rho_i=\rho_i Y_i$.  This implies that the first $r$ factors in the normal form of $(\Delta_X)^m \alpha$ are:
$$
  (\rho_1 \Delta_{Y_1})(\rho_2 \Delta_{Y_2})\cdots (\rho_r \Delta_{Y_r})= (\Delta_X)^r \rho = \rho (\Delta_Y)^r,
$$
where $\rho=\rho_1\cdots \rho_r$. Moreover, $x_i=\Delta_Y$ for $i=r+1,\ldots,m-1$, as we can see in the diagram. And finally we see that  $x_m= \Delta_Y \beta_r$ and
$$
  (\Delta_X)^m \alpha = \rho (\Delta_Y)^{m-1} x_m \cdots x_{m+r} = (\Delta_X)^{m-1} \rho x_m \cdots x_{m+r},
$$
whence $\rho x_m \cdots x_{m+r} = \Delta_X \alpha = \Delta_X \rho \beta = \rho \Delta_Y \beta$
for $\beta = \beta_r x_{m+1}\cdots x_{m+r}$.  In particular, $x_m\cdots x_{m+r} = \Delta_Y \beta$, as we wanted to show.
\end{proof}

Therefore, if we multiply a big power $(\Delta_X)^m$ by some element $\alpha$ which is a product of $r$ simple factors, the normal form of the result still has $m-r-1$ factors of the form $\Delta_Y$ for some~$Y$.

In the forthcoming result, we will need a special procedure to compare elements in~$A_S$. For that purpose, we introduce the following:
\begin{definition}
For every element $\gamma\in A_S$ we will define an integer $\varphi(\gamma)$ as follows: Conjugate $\gamma$ to $\gamma'\in RSSS_\infty(\gamma)$. Let $U=\mbox{supp}(\gamma')$. Then let $\varphi(\gamma)=|\Delta_U|$, the length of the element $\Delta_U$ as a word in the standard generators.
\end{definition}

\begin{proposition}\label{P:varphi(n)}
The integer $\varphi(\gamma)$ is well defined. Moreover, if $\gamma$ is conjugate to a positive element, then $\varphi(\gamma)=|\Delta_X|$, where $X=\mbox{supp}(\beta)$ for any positive element $\beta$ conjugate to~$\gamma$.
\end{proposition}

\begin{proof}
Suppose that $\gamma', \gamma''\in RSSS_\infty(\gamma)$, and let $U=\mbox{supp}(\gamma')$ and $V=\mbox{supp}(\gamma'')$. Then $P_{\gamma'}=A_U$ and $P_{\gamma''}=A_V$, and every element $x$ conjugating $\gamma'$ to $\gamma''$ must also conjugate $A_U$ to~$A_V$. Hence $x^{-1}z_U x = z_V$, by~\autoref{L:parabolic=central_element}.

This implies that $|z_U|=|z_V|$. Moreover, as $A_U$ is conjugate to~$A_V$, we have $z_U=\Delta_U^e$ and $z_V=\Delta_V^e$ for the same $e\in \{1,2\}$. Therefore $|\Delta_U|=|\Delta_V|$, which proves that $\varphi(\gamma)$ is well defined.

On the other hand, if $\gamma$ is conjugate to a positive element $\beta$, and $X=\mbox{supp}(\beta)$, then $P_\beta=A_X$. We can then apply the same argument as above to $\gamma'$ and~$\beta$, to obtain that $\varphi(\gamma)=|\Delta_X|$.
\end{proof}

We can finally show one of the main results in this paper:

\begin{theorem}\label{T:intersection_of_parabolic_is_parabolic}
Let $P$ and $Q$ be two parabolic subgroups of an Artin--Tits group $A_S$ of spherical type. Then $P\cap Q$ is also a parabolic subgroup.
\end{theorem}

\begin{proof}
If either $P$ or $Q$ is equal to $A_S$ or to $\{1\}$, the result is trivially true. Hence we can assume that both subgroups are proper parabolic subgroups.

If $P\cap Q=\{1\}$ the result holds. Hence we will assume that there exists some non-trivial element $\alpha\in P\cap Q$. We take $\alpha$ such that $\varphi(\alpha)$ is maximal (notice that $\varphi(\alpha)$ is bounded above by~$|\Delta_S|$).

Let $P_\alpha$ be the parabolic closure of~$\alpha$. By~\autoref{T:minimal_parabolic_containing_an_element}, we know that $P_\alpha\subseteq P$, and also $P_\alpha\subseteq Q$, so $P_\alpha\subseteq P\cap Q$. Moreover, up to conjugating $P_\alpha$, $P$ and $Q$ by the same suitable element, we can assume that $P_\alpha$ is standard, so $P_\alpha=A_Z\subseteq P\cap Q$ for some $Z\subsetneq S$. Notice that $\Delta_Z\in P_\alpha\subseteq P\cap Q$.

We will show that $P\cap Q=P_\alpha$, that is, $P\cap Q=A_Z$.

Take any element $w\in P\cap Q$. In order to show that $w\in A_Z$, we will consider its parabolic closure~$P_w$, which we will denote by~$T$.  By the above arguments, we have $T\subseteq P\cap Q$ and, in particular, $z_T\in P\cap Q$. Notice that $T$ is conjugate to $A_X$ for some $X\subsetneq S$, hence $z_T$ is conjugate (by the same conjugating element) to the positive element~$z_X$. Since the support of $z_X$ is~$X$, it follows that $P_{z_X}=A_X$, and conjugating back we have $P_{z_T}=T$. Therefore, if we show that $z_T\in A_Z$, this will imply that $T\subseteq A_Z$ and then $w\in A_Z$, as desired.

We then need to show that $z_T\in A_Z$.
Let $a^{-1}b$ be the $np$-normal form of~$z_T$.
We will now construct an infinite family of elements in $P\cap Q$, using $z_T$ and~$\Delta_Z$.
For every $m>0$, consider $\beta_m= z_T (\Delta_Z)^m = a^{-1}b (\Delta_Z)^m$. By construction $\beta_m\in P\cap Q$ for every $m>0$.

Suppose that $a=a_1\cdots a_r$ and $b=b_1\cdots b_s$ are the left normal forms of~$a$ and~$b$, respectively. Then
$$
 \beta_m=  a_r^{-1}\cdots a_1^{-1} b_1\cdots b_s (\Delta_Z)^m.
$$
The $np$-normal form of~$\beta_m$ is computed by making all possible cancellations in the middle of the above expression. As all the involved factors are simple elements, it follows that $\inf(\beta_m)\geq -r$ and $\sup(\beta_m)\leq s+m$.

Recall that $\varphi(\alpha)=|\Delta_Z|$ is maximal among the elements in $P\cap Q$. Let us denote $n=\varphi(\alpha)=|\Delta_Z|$. Now, for every $m>0$, choose some $\widetilde\beta_m \in RSSS_\infty(\beta_m)$.

{\bf Claim:} There is $M>0$ such that $\widetilde\beta_m$ is positive for all $m>M$.

Let $U_m=\mbox{supp}(\widetilde\beta_m)$. We know by maximality of $\varphi(\alpha)$ that $|\Delta_{U_m}|\leq n$. So the length of each simple element in the normal form of~$\widetilde \beta_m$ must be at most~$n$.

Let $x_m^{-1}y_m$ be the $np$-normal form of $\widetilde \beta_m$. As $\widetilde \beta_m\in RSSS_\infty(\beta_m)\subseteq SSS(\beta_m)$, it follows that~$x_m$ is a positive element formed by at most $r$ simple elements, and~$y_m$ is a positive element formed by at most $s+m$ simple elements.

Given some $m>0$, suppose that none of the factors in the left normal form of $y_m$ is equal to $\Delta_{U_m}$. This means that the length of each factor of $y_m$ is at most $n-1$, so $|y_m|\leq (n-1)(s+m)$, that is, $|y_m|\leq (n-1)m+k$ where $k$ is a constant independent of $m$. Hence the exponent sum of $\widetilde\beta_m$ as a product of standard generators and their inverses is $s(\widetilde \beta_m)= |y_m|-|x_m| \leq |y_m| \leq (n-1)m+k$.

But this exponent sum is invariant under conjugation, hence $s(\widetilde \beta_m)=s(\beta_m)=|b|-|a|+nm$. That is, $s(\widetilde \beta_m)=nm+K$ for some constant~$K$ independent of~$m$. We then have $nm+K\leq (n-1)m+k$, that is, $m\leq k-K$.

If we denote $M=k-K$, it follows that for every $m>M$, the left normal form of $y_m$ has some factor equal to $\Delta_{U_m}$. Recall that $U_m=\mbox{supp}(\widetilde \beta_m)$, so $\widetilde \beta_m\in A_{U_m}$. This means that the left normal form of $y_m$ starts with~$\Delta_{U_m}$. But there cannot be cancellations between $x_m^{-1}$ and $y_m$, so this implies that $x_m=1$. Therefore, $\widetilde \beta_m$ is positive for all $m>M$. This proves the claim.
\medskip

We then know that, if $m>M$, the element $\beta_m$ is conjugate to a positive element. The good news is that one can conjugate $\beta_m$ to a positive element $\widehat \beta_m$, using a conjugating element $c_m$ whose length is bounded by a constant independent of~$m$.  Indeed, one just needs to apply iterated cycling to~$\beta_m$ until its infimum becomes non-negative. Since $\inf(\beta_m)\geq -r$, one just needs to increase the infimum at most $r$ times, and we know from~\cite{BKL2} that this can be done with at most $r|\Delta_S|-r$ cyclings. Hence, we can take $\sup(c_m)\leq r|\Delta_S|-r$, say $\sup(c_m)\leq N$, a number which is independent of~$m$.

We then have $\widehat \beta_m= c_m^{-1}\beta_m c_m=c_m^{-1}a^{-1} b (\Delta_Z)^m c_m \in A_S^+$. We will now try to describe the support of the positive element $\widehat\beta_m$.

Consider the element $(\Delta_Z)^m c_m$. By~\autoref{L:Deltas_times_something}, if $m$ is big enough we can decompose $c_m=\rho_m d_m$ so that $\Delta_Z\rho_m = \rho_m \Delta_{Y_m}$ (actually $Z \rho_m = \rho_m Y_m$) for some $Y_m\subsetneq S$, and the left normal form of $(\Delta_Z)^m c_m$ finishes with $m-N-1$ copies of $\Delta_{Y_m}$ followed by some factors whose product equals $\Delta_{Y_m} d_m$. Notice that making $m$ big enough, we can have as many copies of $\Delta_{Y_m}$ as desired.

The negative part of $\widehat\beta_m$ as it is written above is $c_m^{-1}a^{-1}$, which is the product of at most $N+r$ inverses of simple factors. Since $\widehat \beta_m$ is positive, this negative part must cancel completely with the positive part. Namely, it cancels with the first $N+r$ simple factors in the left normal form of $b (\Delta_Z)^m c_m$.

But the first $N+r$ simple factors of $b (\Delta_Z)^m c_m$ are a prefix of~$b$ multiplied by the first $N+r$ factors of $(\Delta_Z)^m c_m$. Recall that we can take $m$ big enough so that $(\Delta_Z)^m c_m$ has as many copies of $\Delta_{Y_m}$ as desired. Hence, for $m$ big enough we can decompose $b(\Delta_Z)^m c_m=A \Delta_{Y_m}d_m$, where $A$ is a positive element containing enough simple factors to absorb $c_m^{-1}a^{-1}$ completely. That is, $A= a c_m B$ for some positive element~$B$.  It follows that $\widehat \beta_m = B \Delta_{Y_m} d_m$.

Now recall that $\rho_m^{-1} \Delta_Z \rho_m =\Delta_{Y_m}$, hence $|\Delta_{Y_m}|=n$. On the other hand, since $\widehat\beta_m$ is positive, its support determines $\varphi(\beta_m)$. Hence, if $U=\mbox{supp}(\widehat \beta_m)$, we have $|\Delta_{U}|\leq n$. As $Y_m\subseteq U$, it follows that $n=|\Delta_{Y_m}|\leq |\Delta_U|\leq n$, so $|\Delta_{Y_m}|= |\Delta_U|=n$ and then $Y_m=U=\mbox{supp}(\widehat \beta_m)$. This implies in particular that $d_m\in A_{Y_m}$.

But now recall that $\widehat \beta_m= c_m^{-1}\beta_m c_m$, where $\widehat \beta_m$ is positive. Hence the minimal parabolic subgroup containing $\beta_m$ is $c_m A_{Y_m} c_m^{-1} = \rho_m d_m  A_{Y_m} d_m^{-1}\rho_m^{-1} = \rho_m A_{Y_m} \rho_m^{-1} = A_Z$.

Therefore, $\beta_m\in A_Z$ for some $m$ big enough. Since $\beta_m= z_T (\Delta_Z)^m$, it follows that $z_T\in A_Z$, as we wanted to show.
\end{proof}

\section{The lattice of parabolic subgroups}

In this secion we will see an interesting simple consequence of the main result in the previous section: The set of parabolic subgroups forms a lattice for inclusion.

\begin{proposition}\label{P:closedPredicate}
Let $A_S$ be an Artin--Tits group of spherical type and let~$\mathcal P$ be the set of parabolic subgroups of~$A_S$.
If~$\pi$ is a predicate on~$\mathcal P$ such that the conjunction of~$\pi(P)$ and~$\pi(Q)$ implies $\pi(P\cap Q)$ for any $P,Q\in \mathcal P$, then the set~$\mathcal P_\pi = \{ P \in\mathcal P : \pi(P) \}$, if nonempty, contains a unique minimal element with respect to inclusion, namely
$$
   \bigcap_{P\in \mathcal P_\pi}{P} .
$$
\end{proposition}

\begin{proof}
First note that for $P,Q\in\mathcal P$, we have $P\cap Q\in\mathcal P$ by \autoref{T:intersection_of_parabolic_is_parabolic}, so $\pi(P\cap Q)$ is defined.

The set~$\mathcal P_\pi$ is partially ordered by inclusion.
Assume that~$\mathcal P_\pi$ is nonempty.
We will show that
$$
   R=\bigcap_{P\in \mathcal P_\pi}{P}
$$
is the unique minimal element in~$\mathcal P_\pi$.
It is clear by definition that~$R$ is contained in every element of~$\mathcal P_\pi$, hence it just remains to show that~$R$ is an element of~$\mathcal P_\pi$.

Notice that the set~$\mathcal P$ of parabolic subgroups in~$A_S$ is a countable set, as every element $P\in \mathcal P$ can be determined by a subset $X\subseteq S$ and an element $\alpha\in A_S$ such that $\alpha^{-1}P\alpha =A_X$.
As there are a finite number of subsets of~$S$ and a countable number of elements in~$A_S$, it follows that~$\mathcal P$ is countable.

Therefore, the set~$\mathcal P_\pi$ is also countable, and we can enumerate its elements: $\mathcal P_\pi = \{P_i : i\in I\}$ for some $I\subseteq \mathbb N$.  Now let
$$
      T_n= \bigcap_{i\in I;\,  i \le n}{P_i}.
$$
By~\autoref{T:intersection_of_parabolic_is_parabolic} and the assumption on~$\pi$, the intersection of any finite number of elements of~¤$\mathcal P_\pi$ is contained in~$\mathcal P_\pi$, so we have $T_n\in \mathcal P_\pi$ for all $n\geq 0$.

We then have the following descending chain of elements of~$\mathcal P_\pi$
$$
   T_0\supseteq T_1 \supseteq T_2 \supseteq \cdots
$$
where the intersection of all the parabolic subgroups in this chain equals~$R$.

We finish the proof by noticing that in~$A_S$ there cannot be an infinite chain of distinct nested parabolic subgroups, as if $\alpha^{-1}A_X\alpha \subsetneq \beta^{-1} A_Y \beta$ then $|X|<|Y|$. Hence, there can be at most $|S|+1$ distinct nested parabolic subgroups in any chain. Therefore, there exists $N\geq 0$ such that $T_N=T_{N+k}$ for every $k>0$, and then
$
    R=\bigcap_{i=0}^{\infty}{T_i}=T_N,
$
so $R$ is an element of~$\mathcal P_\pi$.
\end{proof}

\begin{example}\label{R:minimalContainingParabolic}
Let~$A_S$ be an Artin--Tits group of spherical type and $\alpha\in A_S$.
Applying \autoref{P:closedPredicate} with the predicate $\pi(P) = (\alpha\in P)$, we see that the parabolic closure~$P_\alpha$ of~$\alpha$, which has been shown to be the minimal parabolic subgroup containing~$\alpha$ (\autoref{P:minimalContainingParabolic}), is also the intersection of all parabolic subgroups containing~$\alpha$.
\end{example}

\begin{theorem}\label{T:lattice}
The set of parabolic subgroups of an Artin--Tits group of spherical type is a lattice with respect to the partial order determined by inclusion.
\end{theorem}

\begin{proof}
Let $A_S$ be an Artin--Tits group of spherical type, and let~$\mathcal P$ be the set of parabolic subgroups of~$A_S$. This set is partially ordered by inclusion.
Now assume that $P,Q\in \mathcal P$ are given.

By~\autoref{T:intersection_of_parabolic_is_parabolic}, $P\cap Q$ is the unique maximal parabolic subgroup among those parabolic subgroups contained in both~$P$ and~$Q$.

Applying \autoref{P:closedPredicate} with the predicate $\pi(T) = (P\cup Q\subseteq T)$ shows that there is a unique minimal parabolic subgroup among those parabolic subgroups containing both~$P$ and~$Q$.
Notice that $\mathcal P_{\pi}$ is nonempty as $A_S\in \mathcal P_\pi$.
\end{proof}

\section{Adjacency in the complex of irreducible par\-abolic subgroups}\label{S:annexe}

We postponed to this section the proof of the following result, which characterizes the pairs of adjacent subgroups in the complex of irreducible parabolic subgroups.

{\bf \autoref{centers_commute=3_conditions}}.
Let~$P$ and~$Q$ be two distinct irreducible parabolic subgroups of an Artin--Tits group~$A_S$ of spherical type.
Then $z_Pz_Q=z_Qz_P$ holds if and only if one of the following three conditions is satisfied:
\begin{enumerate}

\item $P\subsetneq Q$.

\item $Q\subsetneq P$.

\item $P\cap Q=\{1\}$ and $xy=yx$ for every $x\in P$ and $y\in Q$.

\end{enumerate}

\begin{proof}
If $P\subsetneq Q$ then $z_P\in Q$ and~$z_Q$ is central in~$Q$, so both elements commute.
Similarly, if $Q\subsetneq P$ then~$z_P$ and~$z_Q$ commute.
Also, if the third condition is satisfied every element of~$P$ commutes with every element of~$Q$, so~$z_P$ and~$z_Q$ commute.

Conversely, assume that~$z_P$ and~$z_Q$ commute.
We can assume $\{1\}\neq P\subsetneq A_S$ and $\{1\}\neq Q\subsetneq A_S$, as otherwise either Condition~1 or Condition~2 holds.
We are going to prove a result which is slightly stronger than what is required:
We shall show that~$P$ and~$Q$ can be \emph{simultaneously} conjugated to standard irreducible parabolic subgroups~$A_X$ and~$A_Y$ (for some subsets $X, Y\subseteq S$); moreover,  one of the following holds:
\begin{enumerate}
\item $X\subsetneq Y$.
\item $Y\subsetneq X$.
\item $X\cap Y=\emptyset$, and all elements of~$X$ commute with all elements of~$Y$.
\end{enumerate}

Notice that the four properties listed in the statement of \autoref{centers_commute=3_conditions} are preserved by conjugation.
Hence, up to conjugation, we can assume that~$P$ is standard.

We decompose~$z_Q$ in $pn$-normal form. Namely, $z_Q=ab^{-1}$ where~$a$ and~$b$ are positive elements such that $a \wedge^{\Lsh} b=1$ (where~$\wedge^{\Lsh}$ means the greatest common suffix in~$A_S$). The suffix order is preserved by right multiplication, so we can right-multiply the above equation by~$b^{-1}$, and we have $(ab^{-1}) \wedge^{\Lsh} 1=b^{-1}$, that is, $z_Q \wedge^{\Lsh} 1=b^{-1}$.

Now~$z_Q$ commutes with~$z_P$, and~$1$ also commutes with~$z_P$. Hence $z_Q z_P z_Q^{-1}$ and $1 z_P (1)^{-1}$ are positive elements (we are assuming that~$P$ is standard, so~$z_P$ is positive). It follows by convexity (\autoref{L:convexity} applied to the suffix order) that $(z_Q \wedge^{\Lsh} 1)z_P (z_Q \wedge^{\Lsh} 1)^{-1}$ is positive, that is, $b^{-1}z_P b$ is positive.

But we know from \autoref{C:conjugators_of_positive_elements} (see also \cite[Corollary~1]{Cum}) that each positive conjugate of~$z_P$ is the generator of the center of a proper standard irreducible parabolic subgroup. That is, $b^{-1}z_P b=z_X$ for some $X\subsetneq S$.

On the other hand, it is shown in~\cite[Theorem~3]{Cum} that if $z_Q=ab^{-1}$ is in $pn$-normal form, then~$b$ is the minimal standardizer of~$Q$, that is,~$b$ is the smallest positive element that conjugates~$Q$ to a standard parabolic subgroup, so $b^{-1}z_Qb = z_Y$ for some $Y\subsetneq S$.

Therefore, when conjugating both~$z_P$ and~$z_Q$ by~$b$, we obtain elements~$z_X$ and~$z_Y$, generators of the centers of proper standard parabolic subgroups. We can then assume, up to conjugacy, that~$P$ and~$Q$ are both proper standard irreducible parabolic subgroups.

Now we will need the following:

{\bf Claim:} Let~$A_S$ be an arbitrary Artin--Tits group. Let $s_0,\ldots, s_k\in S$ be standard generators such that~$s_i$ does not commute with~$s_{i+1}$, and $s_i\neq s_{i+2}$ for every~$i$. If an element $\alpha\in A_S^+$ is represented by a positive word~$w$ which contains the subsequence $s_0s_1\cdots s_k$, then all positive words representing~$\alpha$ contain the same subsequence.

{\it Proof of the claim:}  It is known~\cite{Par} that Artin monoids inject in their groups. This implies that every positive word representing~$\alpha$ is obtained from~$w$ after a finite sequence transformations, each one replacing a subword~$sts\cdots$ (having~$m(s,t)$ letters) with~$tst\cdots$ (also having~$m(s,t)$ letters). It suffices to show that the word obtained from~$w$ after a single transformation contains the subsequence $s_0\cdots s_k$. If $m(s,t)=2$, the transformation replaces~$st$ with~$ts$. But the subword~$st$ can intersect the subsequence $s_0\cdots s_k$ in at most one letter (as~$s_i$ and~$s_{i+1}$ do not commute for every~$i$), hence the subsequence survives after the transformation. If $m(s,t)\geq 3$ then the subword $sts\cdots$ intersects the subsequence $s_0\cdots s_k$ of~$w$ in at most two consecutive letters (as $s_i\neq s_{i+2}$ for every~$i$). This intersection is either $(s_{i},s_{i+1})=(s,t)$ or $(s_{i},s_{i+1})=(t,s)$. In either case, the subsequence survives after the transformation, as~$tst\cdots$ contains both possible subsequences. This shows the claim.
\medskip

Recall that we are assuming that~$P$ and~$Q$ are distinct nonempty proper irreducible standard parabolic subgroups of~$A_S$, so $P=A_X$ and $Q=A_Y$, where $X,Y\subsetneq S$ and~$\Gamma_X$ and~$\Gamma_Y$ are connected graphs. We are also assuming that~$z_P$ and~$z_Q$ (that is,~$z_X$ and~$z_Y$) commute. We will further assume that none of the three conditions in the statement holds, and we will arrive at a contradiction.

If $X\cap Y=\emptyset$, Condition~3 not being satisfied implies the existence of $a\in X$ and $b\in Y$ that are adjacent in~$\Gamma_S$.

Otherwise, as Condition~1 is not satisfied and~$\Gamma_X$ is connected, there exist $a\in X\setminus Y$ and $s_1 \in X\cap Y$ that are adjacent in~$\Gamma_S$.
Moreover, as Condition~2 is not satisfied and~$\Gamma_Y$ is connected, there are $b\in Y\setminus X$ and a simple path $s_1,s_2,\ldots,s_k=b$ in~$\Gamma_Y$.

In either case, we have a path $a=s_0,s_1,\ldots,s_k=b$ satisfying the hypothesis of the above claim, where $a\in X\setminus Y$, $b\in Y\setminus X$, and $s_1,\ldots,s_k\in Y$.


Now consider the element~$z_Pz_Q$, which is equal to~$z_Xz_Y$. It is a positive element, and any representative of~$z_X$ involves the letter~$a$. On the other hand, let us denote $A_i=\{s_1,\ldots,s_i\}$ for $i=1,\ldots,k$. Then
$\Delta_{A_k} \preccurlyeq \Delta_Y \preccurlyeq z_Y$, and we have a decomposition
$$
   \Delta_{A_k}=\Delta_{A_1}r_{A_1,s_2}r_{A_2,s_3}\cdots r_{A_{k-1},s_k}
$$
where the product of the~$i$ leftmost factors is precisely~$\Delta_{A_i}$, for $i=1,\ldots,k$. Now notice that $s_1=\Delta_{A_1}$ and recall that~$s_i$ is the first letter of $r_{A_{i-1},s_i}$ (by \autoref{L:ribbon_prefix}) for $i=2,\ldots,k$. Therefore, the sequence $s_0,\ldots,s_k$ is a subsequence of~$z_Xz_Y$.

From the above claim, it follows that every positive word representing~$z_Xz_Y$ must contain $s_0,\ldots,s_k$ as a subsequence. Now choose a representative of~$z_Yz_X$ which is the concatenation of a word representing~$z_Y$ and a word representing~$z_X$. In such a representative, each instance of the letter~$b$ appears to the left of each instance of the letter~$a$. Therefore, this word does not contain $s_0,\ldots,s_k$ as a subsequence. Hence $z_Xz_Y\neq z_Yz_X$, that is, $z_Pz_Q\neq z_Qz_P$.
The latter is a contradiction which finishes the proof.
\end{proof}

\bibliographystyle{plain}
\bibliography{parabolic_subgroups}

{\bf Mar\'{\i}a Cumplido.}\\
maria.cumplido.cabello@gmail.com\\
Univ Rennes, CNRS, IRMAR - UMR 6625, F-35000 Rennes (France).\\
\& Depto. de \'Algebra. Instituto de Matem\'aticas (IMUS). \\
Universidad de Sevilla. Av. Reina Mercedes s/n, 41012 Sevilla (Spain).

\medskip

{\bf Volker Gebhardt.}\\
 v.gebhardt@westernsydney.edu.au\\
 Western Sydney University \\
 Centre for Research in Mathematics\\
 Locked Bag 1797, Penrith NSW 2751, Australia

\medskip

{\bf Juan Gonz\'alez-Meneses.}\\
meneses@us.es\\
Depto. de \'Algebra. Instituto de Matem\'aticas (IMUS). \\
Universidad de Sevilla.  Av. Reina Mercedes s/n, 41012 Sevilla (Spain).

\medskip

{\bf Bert Wiest.}\\
bertold.wiest@univ-rennes1.fr\\
Univ Rennes, CNRS, IRMAR - UMR 6625, F-35000 Rennes (France).

\end{document}